\documentclass[11pt]{amsart}
\usepackage{amssymb,amsmath,amsthm,amscd,mathrsfs,graphicx, color}
\usepackage[cmtip,all,matrix,arrow,tips,curve]{xy}
\usepackage[active]{srcltx}

\numberwithin{equation}{section}
\newtheorem{teo}{Theorem}[section]
\newtheorem{pro}[teo]{Proposition}
\newtheorem{lem}[teo]{Lemma}
\newtheorem{cor}[teo]{Corollary}

\theoremstyle{definition}

\newtheorem{exa}[teo]{Example}

\theoremstyle{remark}
\newtheorem{rem}[teo]{Remark}

\newtheorem{cnv}{}

\begin{document}
\bibliographystyle{amsplain}

\title[Stable reduction]{A tour of stable reduction with applications}
\author[Casalaina-Martin]{Sebastian Casalaina-Martin }
\address{University of Colorado at Boulder, Department of Mathematics,  Campus Box 395,
Boulder, CO 80309-0395, USA}
\email{casa@math.colorado.edu}

\thanks{The  author was partially supported by NSF grant DMS-1101333.}

\subjclass[2010]{14H10, 14J10, 14K10}

\dedicatory{To Joe Harris.}

\date{\today}
\begin{abstract} 
The stable reduction theorem for curves  asserts  that for a family of stable curves over the punctured disk,  after a finite base change, the family can be completed in a unique way to a family of stable curves over the disk.   In this survey we discuss stable reduction theorems in a number of different contexts.  This includes a review of recent results on abelian varieties, canonically polarized varieties, and singularities.  We also consider the  semi-stable reduction theorem and  results concerning simultaneous stable reduction.
\end{abstract}
\maketitle


\section*{Introduction}

The stable reduction theorem for curves \cite{dm,woodshole}  asserts that given a family of stable curves over the punctured disk,  after a finite base change, the family can be completed in a unique way to a family of stable curves over the disk.  In particular, the central fiber of the new family is determined, up to isomorphism, by the original family.     This theorem plays a central role in the study of curves.  A consequence is  the fundamental 
result that the moduli space of stable curves is compact.  Qualitatively,  the theorem provides control over degenerations of smooth curves: when studying one-parameter degenerations, one may restrict to the case where the limit has normal crossing singularities.  

In \cite[\S 3.C]{hm} Harris--Morrison give a beautiful treatment of the stable reduction  theorem from a computational perspective.  They outline a proof of the theorem that provides the reader with a method of completing this process in particular examples, and importantly, of identifying the central fiber of the new family.    The aim of this survey is to complement  \cite[\S 3.C]{hm} with stable reduction problems in other settings.

 Roughly speaking, by a stable reduction problem we mean the problem of determining a class of degenerations so that  a family over the punctured disk can, after a finite base change,  be extended in a unique way to a family over the disk. 
 
Typically the motivation will be a moduli problem, where one is given a particular class of geometric or algebraic objects that determine a non-compact moduli space.  The stable reduction problem can be viewed as providing a modular compactification of the moduli space.
 In the language of stacks, stable reduction  is equivalent to the valuative criterion of  properness for the moduli stack (see \S\ref{secsrstacks}).  In \S \ref{secexa}, we work out  an explicit motivating example.  
  The main cases we will consider in the survey are  stable reduction for abelian varieties \S \ref{secab}, curves \S\ref{seccurves}, and canonically polarized varieties \S \ref{secsrhd}.

Determining a class of degenerations that will provide a stable reduction theorem is often difficult.  
  In this situation one can begin with the qualitative goal of obtaining some level of control over degenerations.   For instance, one may focus on restricting the singularities of the central fiber, or controlling invariants such as monodromy \S \ref{secmono}.    
  
  In this direction, semi-stable reduction is the problem of filling in (possibly after a finite base change) a smooth family of schemes over the punctured disk  to a family where the total space is smooth, and the central fiber is a reduced scheme with simple normal crossing singularities.    Unlike the case of stable reduction, such a completion will not be unique.  On the other hand, the singularities (and topology) of a semi-stable reduction will typically be much simpler than that of a stable reduction.    
  
  The main result  in this context is 
 a theorem of Mumford et al.~\cite{mumetal}  stating that  semi-stable reductions exist in complete generality in characteristic $0$ (see \S \ref{secssr}). 
This plays a central role in many stable reduction theorems.  In particular, Koll\'ar--Shepherd-Barron--Alexeev have developed an approach to the stable reduction problem using log canonical models of semi-stable reductions.  We use the case of curves \S \ref{seccurves} and canonically polarized varieties \S\ref{seckol} to discuss this.

One can also consider the question of extending families over higher dimensional bases.  Given a stable reduction theorem one can then ask whether families over a dense open subset of a scheme of dimension $2$ or more can be extended after a generically finite base change.  We call this a simultaneous stable reduction problem.    For moduli spaces that are proper Deligne--Mumford stacks, it is well known that simultaneous stable reductions always exist (Theorem \ref{teossrfd}, \cite{fedorchuk}, \cite{edidinetal}).  However, in general, this is a delicate problem.  Explicitly describing such a generically finite base change can be quite difficult.   

We review simultaneous stable reduction in \S \ref{secsimsr}, where we focus on the cases of abelian varieties and curves.  
 One recent motivation for considering this type of problem has to do with resolving birational maps between moduli spaces.  The cases arising in  the Hassett--Keel pogram for the moduli space of curves  have received a great deal of attention recently; we review this in \S \ref{secade}.

In light of the breadth of the topic, to prevent this survey from becoming too lengthy, we have chosen to focus on  a  few cases that  have a historic connection to the stable reduction theorem for curves, capture the flavor of the topic in general, and which point to some of the recent progress in the field.   
We also include a number of examples.   
In the end, the material chosen reflects the author's  exposure to the subject, and he apologizes to those people whose work was not included.

\subsection*{Acknowledgements} 
 It is a pleasure to thank J.~Achter,   J.~Alper, J.~de Jong, M.~Fedorchuk,  D.~Grant,  C.~Hall, D.~Jensen, J.~Kass,  J.~Koll\'ar, S.~Kov\'acs,  R.~Laza,  M.~Lieblich, Z.~Patakfalvi, R.~Smith, R.~Virk and J.~Wise for discussions about various topics covered in the survey that have greatly improved the exposition.  Special thanks are also due to  J.~Achter,  J.~Alper,    D.~Jensen, J.~Kass,   J.~Koll\'ar,  S.~Kov\'acs,   R.~Laza,  Z.~Patakfalvi, R.~Smith, R.~Virk and the referee for detailed comments on  earlier drafts.

\subsection*{Notation and conventions}

\begin{cnv}
A \textbf{family of schemes} $f:X\to B$ will be a flat, surjective, finite type   morphism of schemes, of constant relative dimension.   The scheme $B$ will be called the \textbf{base}  of the family and $X$ the \textbf{total space} of the family.   For a point $b\in B$,  we denote by $X_b$ the fiber of $f$ over $b$.   
\end{cnv}

\begin{cnv}
We will typically use the following notation for spectrums of discrete valuation rings (DVRs).  For a DVR $R$ we will use the notation $K=K(R)$ for the fraction field, and $\kappa=\kappa(R)$ for the residue field. We will set $S=\operatorname{Spec} R$, with generic point $\eta=\operatorname{Spec}K$ and closed point $s=\operatorname{Spec}\kappa$.     
\end{cnv}

\begin{cnv}
If  $B$ is noetherian, and the family $f:X\to B$ is of constant relative dimension $d$, the \textbf{discriminant}, denoted $\Delta$, is the $0$-th Fitting scheme of the push-forward of the structure sheaf of the $d$-th Fitting scheme of the coherent sheaf $\Omega_{X/B}$.    The discriminant parameterizes the  singular fibers of the family.  Typically, we consider this in the case where either the $d$-th fitting scheme of $\Omega_{X/B}$ is finite over $B$, or $f$ is proper; in these  cases $\Delta$ is a closed subscheme of $B$.
\end{cnv}

\begin{cnv}
Let  $X$ be a scheme over an algebraically closed field $k$,  which is regular in codimension one,  and let $D$ be an  effective Weil divisor  on $X$.  We say $D$  is in \textbf{\'etale (resp.~Zariski or simple) normal crossing} position  if   $X$ is regular along the support of $D$ and  for each closed point $x\in \operatorname{Supp}(D)$ there exists an \'etale morphism (resp.~an open inclusion) $f:U\to X$ such that for any $u\in U$ with $f(u)=x$, there is a local system of parameters $u_1,\ldots,u_n$ for $\mathscr O_{U,u}$ so that the pull-back of $D$ via the composition $\operatorname{Spec}\mathscr O_{U,u}\to U\to X$, is  defined by 
a product $u_1^{n_1}\cdots u_r^{n_r}$, for some $0\le r\le n$ and some non-negative integers $n_1,\ldots,n_r$.  
We will say a divisor is \textbf{nc}, (resp.~\textbf{snc}) if it is in  \'etale normal crossing (resp.~simpe normal crossing) position.
\end{cnv}

\begin{cnv}
A \textbf{modification} is a 
proper, birational morphism.  An \textbf{alteration}  is a  generically finite, proper, surjective morphism.  
\end{cnv}

\begin{cnv}
A \textbf{germ} will be the spectrum of a complete local ring $A$ and we will use the notation $(X,x)$ with $X=\operatorname{Spec}A$  and $x$ the maximal idea of $A$. A 
\textbf{(germ of a) singularity}  will be a germ that is singular at $x$.  
  We will typically focus on \textbf{hypersurface singularities}; by this we mean the case where  $X=\operatorname{Spec}k[[x_1,\ldots,x_n]]/(f)$, $f\in k[[x_1,\ldots,x_n]]$ and $k$ is a field.  
We will say a singularity is \textbf{isolated} if $\mathscr O_{X,x'}$ is a regular local ring for all $x'\in X$ with $x'\ne x$.    
\end{cnv}

\section{An example via elliptic curves} \label{secexa}

In this section we start by briefly reviewing a stable reduction for a $1$-parameter family of elliptic curves degenerating to a cuspidal cubic, as  described in Harris--Morison \cite[Ch.~3.C]{hm}.  We then turn to the case of simultaneous stable reduction, and give an explicit computation of a simultaneous stable reduction for a versal deformation of a cuspidal cubic.  In other words, we  analyze all $1$-parameter degenerations at once.    
The presentation we give is a special case of a larger example described by Laza and the author in \cite{cml} (related to well known work of  Brieskorn   \cite{bruber} and Tyurina \cite[\S3]{tjurina0}; 
see also the recent work of  Fedorchuk  \cite[\S 5]{maksym}, \cite{fedorchuk}) and can be viewed as an  extension of the discussion in  Harris--Morrison \cite[p.129-30]{hm}.

\subsection{Stable reduction for a pencil of cubics}

Stable reduction concerns one parameter degenerations.  In this subsection we briefly review a stable reduction for a  family of non-singular plane cubics degenerating to a plane cubic with a cusp.    
For brevity, we will leave out any computations.   The reader is encouraged to read  \cite[Ch.~3.C]{hm}, where this example is worked out in detail (see also \S \ref{secandirect} and \S\ref{secmonocusp}).

Fix an algebraically closed field $k$ with characteristic not equal to $2$ or $3$ and consider the family $X\to B=\mathbb A^1_k$ given by:
$$
x_2^2+x_1^3+t_3=0,
$$
where $t_3$ is the parameter on $\mathbb A^1_k$.
The family is smooth away from $t_3=0$, and the central fiber has a cusp.  

\begin{figure}[htp]\label{figcuspfam}
\includegraphics[scale=0.25]{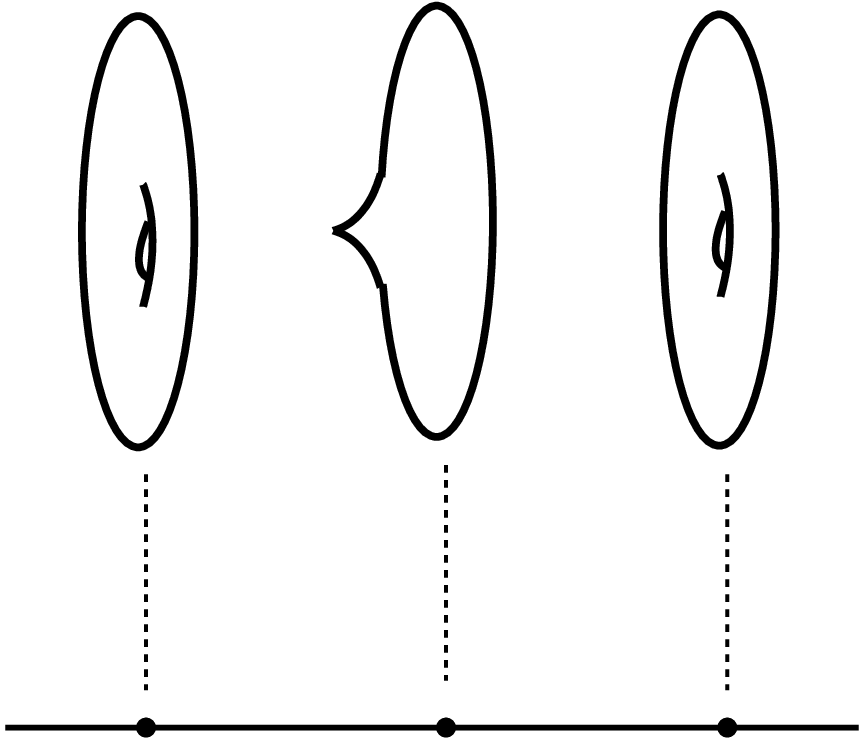}
\caption{A degenerate family}
\end{figure}

The goal of  stable reduction for curves is  to replace the central fiber of $X\to B$ with a  stable curve.  We will see via the theory of monodromy, and by a direct computation, that this is not possible without a degree six base change.    
 So  let $B'=\operatorname{Spec}k[t_3']\to \mathbb A^1_k$ be the degree six map given by $t_3'\mapsto (t_3')^6$.   After  base change we obtain a new  family $X'=B'\times_BX\to B'$, which is also smooth away from the central fiber, and has a cuspidal cubic as the central fiber. Let $U=\operatorname{Spec}k[t_3']_{t_3'}$.    
 
By an appropriate sequence of birational transformations of the surface $X'$ (e.g.~\cite[p.122-129]{hm} for a slight variation), one can obtain a new family $\widehat X\to B'$ that is isomorphic to $X'$ over $U$, and such that the central fiber $\widehat X_0$ of $\widehat X$ is a non-singular curve.  
\begin{figure}[htp]\label{figcuspfamsr}
\includegraphics[scale=0.25]{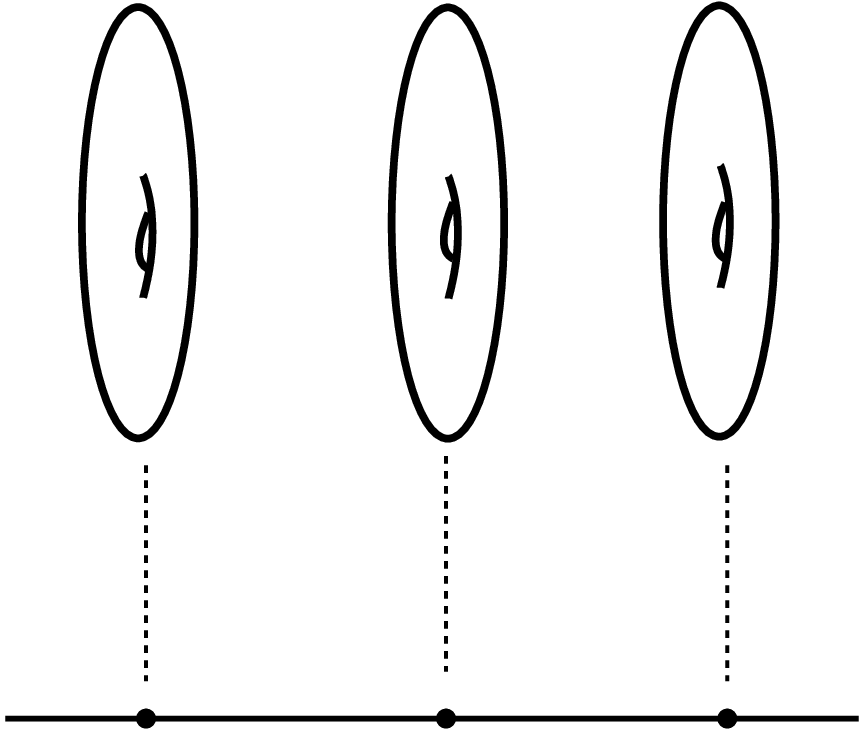}
\caption{The stable reduction}
\end{figure}
The family $\widehat X\to B'$ is called a stable reduction of the family $X\to B$.   
For an explicit computation with  equations, see 
 \S\ref{secmonocusp}  (and also \S \ref{secandirect}).

\subsection{A $2$-parameter family of cubics}  
We will next consider a concrete example of simultaneous stable reduction.    While in general this is a more delicate question than that of stable reduction, in this example we will be able to make the simultaneous stable reduction completely explicit.    We start by describing the family of curves, which can be viewed as a versal deformation of a cuspidal cubic.

Fix an algebraically closed field $k$ with characteristic not equal to $2$ or $3$.  
Consider the family  of curves 
$$
x_2^2+x_1^3+t_2x_1+t_3=0
$$
with parameters $t_2$ and $t_3$.   We denote  the family by $X\to B$.   One can easily check that the curve defined by the point $(t_2,t_3)$ is non-singular if and only if  $4t_2^3-27t_3^2\ne 0$.  The curve has a unique singularity, which is a node, if $4t_2^3-27t_3^2=0$ and $(t_2,t_3)\ne (0,0)$, and the curve has a unique singularity, which is a cusp, if $(t_2,t_3)=(0,0)$.
Despite the family technically being none of the following, we will view it simultaneously as a family of projective curves of arithmetic genus one, a degenerate family of abelian varieties, and  a deformation of a cusp.

\begin{figure}[htp]\label{figdegen}
\includegraphics[scale=0.25]{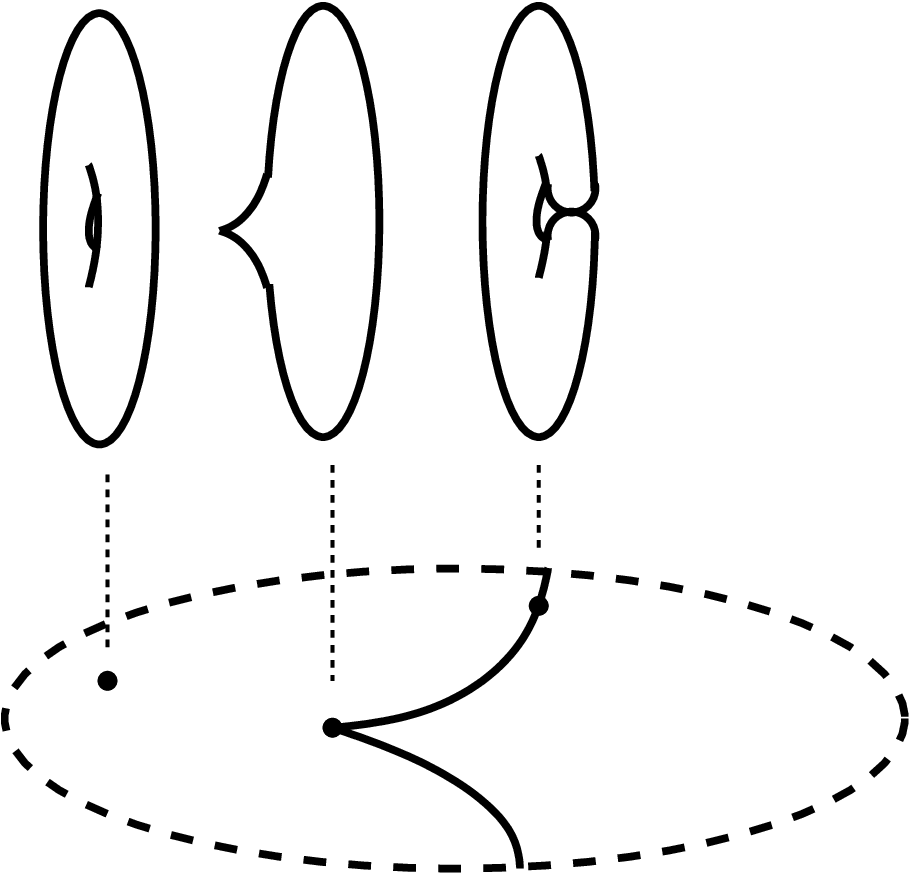}
\caption{A degenerate family}
\end{figure}

\begin{rem}
To make this discussion precise we should take
$$
 X=\operatorname{Proj}_{\mathbb A^2_k}\left(\frac{k[t_2,t_3][X_0,X_1,X_2]}{(X_0X_2^2+X_1^3+t_2X_0^2X_1+X_0^3t_3)}\right)\subseteq \mathbb P^2_k\times \mathbb A_k^2,
$$
$B=\mathbb A_k^2=\operatorname{Spec}k[t_2,t_3]$,  $\pi: X\to \mathbb A_k^2$ the morphism induced by the  second projection $\mathbb P^2_k\times \mathbb A^2_k\to \mathbb A^2_k$, and $\sigma_\infty:\mathbb A^2_k\to  X$   the section at infinity given by the ring homomorphism
$$
\frac{k[t_2,t_3][X_0,X_1,X_2]}{(X_0X_2^2+X_1^3+t_2X_0^2X_1+X_0^3t_3)}\to k[t_2,t_3][X_0]
$$
defined by the ideal $(X_1,X_2)$.  
We then obtain a diagram
\begin{equation} 
\xymatrix
{ 
 X \ar@{^(->}[r] \ar@{->}[rd]^\pi& \mathbb P^2\times \mathbb A^2_k\ar@{->}^{\pi_2}[d]\\
& \mathbb A_k^2 \ar@{->}@/^1.5pc/[lu]^{\sigma_\infty}.\\
}
\end{equation}
The morphism $\pi: X\to \mathbb A_k^2$ is a flat family of projective curves of arithmetic genus one.  The section $\sigma_\infty$ defines a group scheme structure and polarization on the generic fiber.  This makes the generic fiber a principally polarized abelian scheme of dimension one.   Restricting to germs, we obtain a deformation of a cusp.
\end{rem}

Let us make a few more informal observations.   Set 
$$
 G= X-\{(0,0,t_2,t_3): 4t_2^3-27t_3^2=0\}
$$
(where here we are taking $X$ to be the projective family). Then $\pi: G\to \mathbb A^2_k$ is a family of commutative  groups.  The group  parameterized by $(t_2,t_3)$ is a copy of $\mathbb G_m$, if $4t_2^3-27t_3^2=0$ and $(t_2,t_3)\ne (0,0)$.  The group is a copy of $\mathbb G_a$, if $(t_2,t_3)=(0,0)$.   These are the groups of line bundles of degree zero on the corresponding fibers.  In fact $G/\mathbb A^2_k$ is the relative (connected component of the) Picard scheme $\mathbf{Pic}^0_{ X/\mathbb A^2_k}$, and  $ X/\mathbb A^2_k$ is  the compactified  (connected component of the) Picard scheme  $\overline{\mathbf{Pic}}^0_{X/\mathbb A^2_k}$ (see Altman--Kleiman \cite{ak80}).

For the purpose of this discussion, we view it as pathological that the central fiber of the family $\pi:X\to \mathbb A_k^2$ has a cusp (and the central fiber of the family $\mathbf{Pic}^0_{X/\mathbb A_k^2}\to \mathbb A^2_k$ is an additive group).  Our goal will be to modify the family so that we may replace the central fiber with a nodal curve (or a copy of $\mathbb G_m$ in the case of the family of groups, or a collection of smooth components meeting transversally in the case of a singularity).

The problem can also be stated in stack-theoretic language.  Let $\overline {\mathcal M}_{1,1}$  be the moduli stack of Deligne--Mumford stable, one-pointed curves of arithmetic genus one, and let $\overline M_{1,1}$  be the coarse moduli space.    The family $X/\mathbb A_k^2$ defines a rational map $ \mathbb A^2_k\dashrightarrow \overline{\mathcal M}_{1,1}$ and we would like to give a resolution of this map.

\subsection{Explicit simultaneous stable reduction}\label{secandirect}
We now construct an explicit simultaneous stable reduction of the family.  We will do this in several steps, and then  
discuss a monodromy computation that sheds light on the problem.

\subsubsection{Step 1: pulling back by a Weyl (group) cover}\label{secwc}

Consider the map
$$
\{a_1+a_2+a_3=0\}\to \mathbb A_k^2
$$
$$
(a_1,a_2,a_3)\mapsto (a_1a_2+a_1a_3+a_2a_3,-a_1a_2a_3).
$$ 
The families obtained are given by the diagram below.

\begin{equation} 
\xymatrix
{ 
\{a_1+a_2+a_3=x_2^2+\prod_{i=1}^3(x_1-a_i)=0\}   \ar@{->}[d]   \ar@{->}[r] &  \{x_2^2+x_1^3+t_2x_1+t_3=0\} \ar@{->}[d]\\
\{a_1+a_2+a_3=0\}  \ar@{->}[r]&\mathbb A_k^2.
}
\end{equation}

\begin{figure}[htp]\label{figwc}
\includegraphics[scale=0.25]{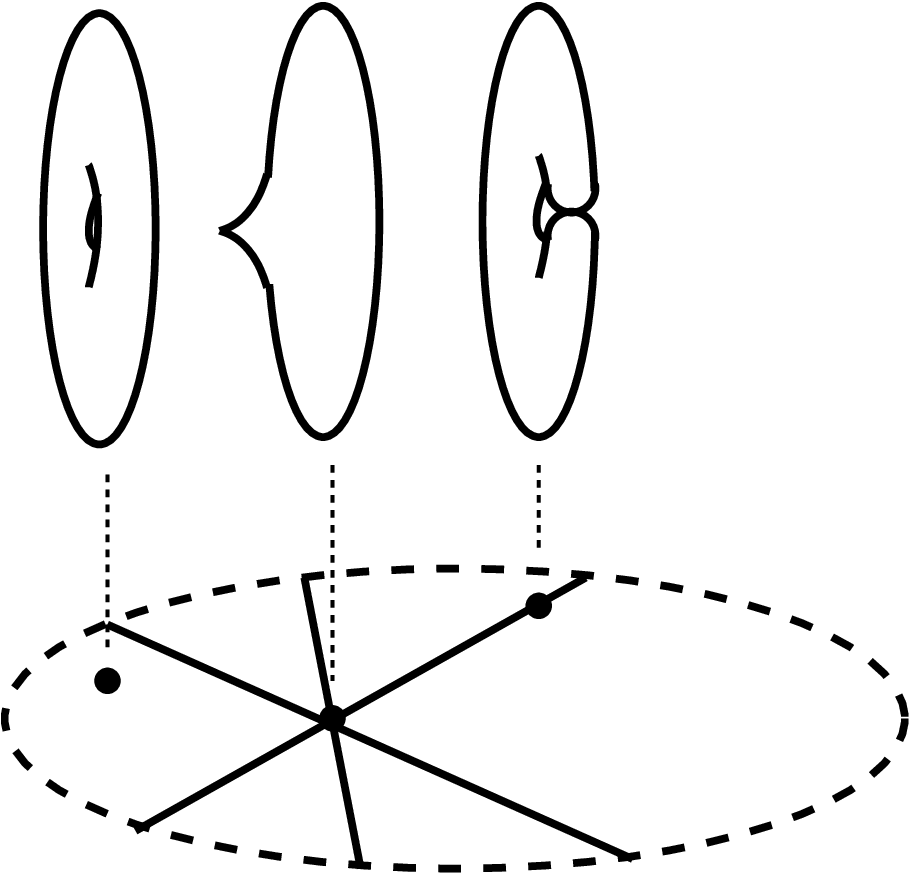}
\caption{The Weyl cover}
\end{figure}

There is still a unique fiber that is cuspidal, but the discriminant has been replaced by a hyperplane arrangement of type $A_2$, given by the equation
$$
\{(a_2-a_3)^2(a_1-a_3)^2(a_1-a_2)^2=0\}.
$$
Set  $B'\to B$ to be the finite (Weyl) cover defined above, and set $X'\to B'$ to be the family obtained by pull-back.  The Weyl group in this case is the group of type $A_2$; i.e.~the permutation group $\Sigma_3$.

\subsubsection{Step 2: a wonderful blow-up} \label{secwbu}
It is a general principle that putting the discriminant locus  into normal crossing position is beneficial (not only is a normal crossing divisor easier to understand, there is also the  Borel Extension Theorem  \cite{borelextension} for abelian varieties and the work of de~Jong--Oort \cite{dejongoort} and Cautis \cite{cautis} for stable curves, all of which will be discussed in more detail in \S \ref{secab} and \S \ref{seccurves}).  

We put the discriminant in this example into nc position  by blowing up the point that is the intersection of its three components.  Explicitly, on one coordinate patch,  we consider the map 
$$
\{1+b_2+b_3=0\} \to \{a_1+a_2+a_3=0\}
$$
$$
(b_1,b_2,b_3)\mapsto (b_1,b_1b_2,b_1b_3).
$$
Pulling the family back by this map gives the new family:
\begin{equation} 
\xymatrix
{ 
\{1+b_2+b_3=x_2^2+(x_1-b_1)\prod_{i=1}^2(x_1-b_1b_i)=0\} 
    \ar@{->}[d] \ar@{->}[r] &\ldots \\
  \{1+b_2+b_3=0\} \ar@{->}[r] &\ldots.
}
\end{equation}

\begin{figure}[htp]\label{figwbu}
\includegraphics[scale=0.25]{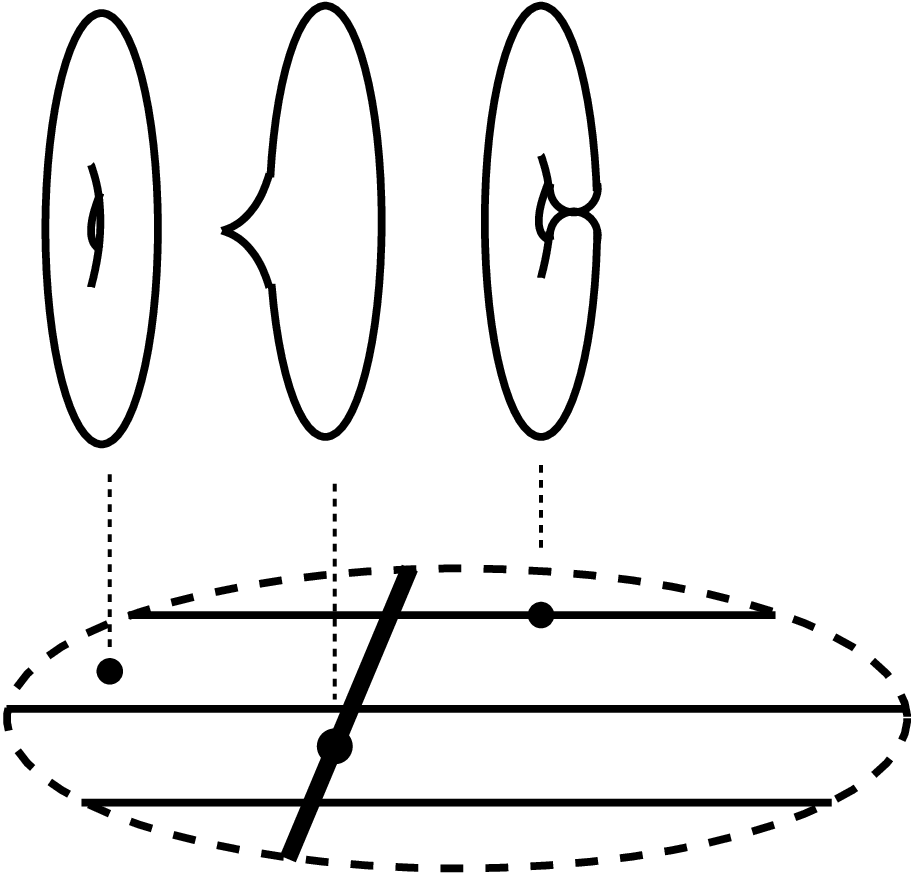}
\caption{The wonderful blow-up}
\end{figure}

 Denote the space obtained by this blow-up as $\tilde B\to B'$.   We call this the wonderful blow-up. Let $\widetilde X\to \widetilde B$ be the family obtained by pull-back.  This restricts to  a family of cuspidal curves over the exceptional curve $\{b_1=0\}$.  Note that the generic point of each irreducible component of the discriminant now parameterizes curves with a unique  singularity of type $A_1$ or $A_2$.  

\begin{rem}
We will see below in \S \ref{secnegex} that over $\tilde B$ we can not replace the cuspidal curves with stable curves.  In other words, there does not exist a morphism extending the rational map $\tilde B \dashrightarrow \overline {\mathcal M}_{1,1}$.  However, we can extend the map to the  moduli scheme (see \S \ref{secsum}); i.e.~there is a morphism extending the rational map $\tilde B\dashrightarrow \overline M_{1,1}$.  
\end{rem}

\subsubsection{Step 3: a double cover}\label{secdc}   In order to obtain a family of stable curves, we will need to take a double cover of the base, branched along the exceptional locus.     The double cover is not possible globally (the exceptional divisor does not admit a square root), so we proceed  locally.  Consider the map
$$
\{1+c_2+c_3=0\} \to \{1+b_2+b_3=0\}
$$
$$
(c_1,c_2,c_3)\mapsto (c_1^2,c_2,c_3).
$$
Pulling the family back by this map gives the new family:
\begin{equation} 
\xymatrix
{ 
\{1+c_2+c_3=x_2^2+(x_1-c_1^2)\prod_{i=1}^2(x_1-c_1^2c_i)=0\} 
    \ar@{->}[d] \ar@{->}[r] &\ldots \\
  \{1+c_2+c_3=0\} \ar@{->}[r] &\ldots.
}
\end{equation}
 Let us denote this finite cover by $\tilde B'\to\tilde B$ and let $\widetilde X'\to\widetilde B'$ be the family obtained by pull-back.

\subsubsection{Step 4: blowing up the cusp locus in the total space}\label{secbucusps} 

There is a family of cuspidal curves lying over the locus $\{c_1=0\}$.    In the total space $\widetilde X'$, the locus of cusps in the fibers is given as $\{c_1=x_1=x_2=0\}$.  Our goal will be to perform a blow-up supported on this locus that will provide a family of semi-stable curves.

To do this, blow-up $\widetilde X'$ along the ideal
$$
I=\left((c_1^2,x_1)^{3}, (c_1^{3},c_1x_1)\cdot(x_2), x_2^2\right).
$$
Let us denote the resulting family as $\operatorname{Bl}_I\widetilde X'\to \widetilde B'$.  
The blow-up replaces the cuspidal curves with nodal curves  consisting of two irreducible components: the desingularization of the cuspidal curve, which is a copy of  $\mathbb P^1$ sitting inside of the blow-up $\operatorname{Bl}_{(x_1^{3},x_2^2)}\mathbb A^2_k$, and a stable elliptic curve sitting inside of the weighted projective space $\mathbb P(1,2,3)$.   We mention here that 
Hassett \cite[\S 6.2]{hassettstable} has determined the tails arising from a much more general class of singularities; we will discuss these results later in \S \ref{secsing}.

In short, we have locally (on the base) constructed an explicit semi-stable reduction of the cuspidal family, which is stable except in the fibers over the locus $\{c_1=0\}$, where it is nodal, but not stable.

\subsubsection{Step 5: the relative dualizing sheaf} \label{secreldual} 
Finally, one can take the relative canonical model (obtained via the relative dualizing sheaf) for the family of  nodal curves.  Concretely, this will contract the extraneous $\mathbb P^1$s in the fibers, giving a family of stable curves.  Let us denote this family by $\widehat {X}\to \widetilde{B}'$.

\subsubsection{Summary} \label{secsum}
We now have a family $\widehat {X}\to \widetilde{B}'$ of stable curves extending the  pull-back of the original family, where $\widetilde{B}'$ is an alteration of $B$.  
\begin{figure}[htp]\label{figwsr}
\includegraphics[scale=0.25]{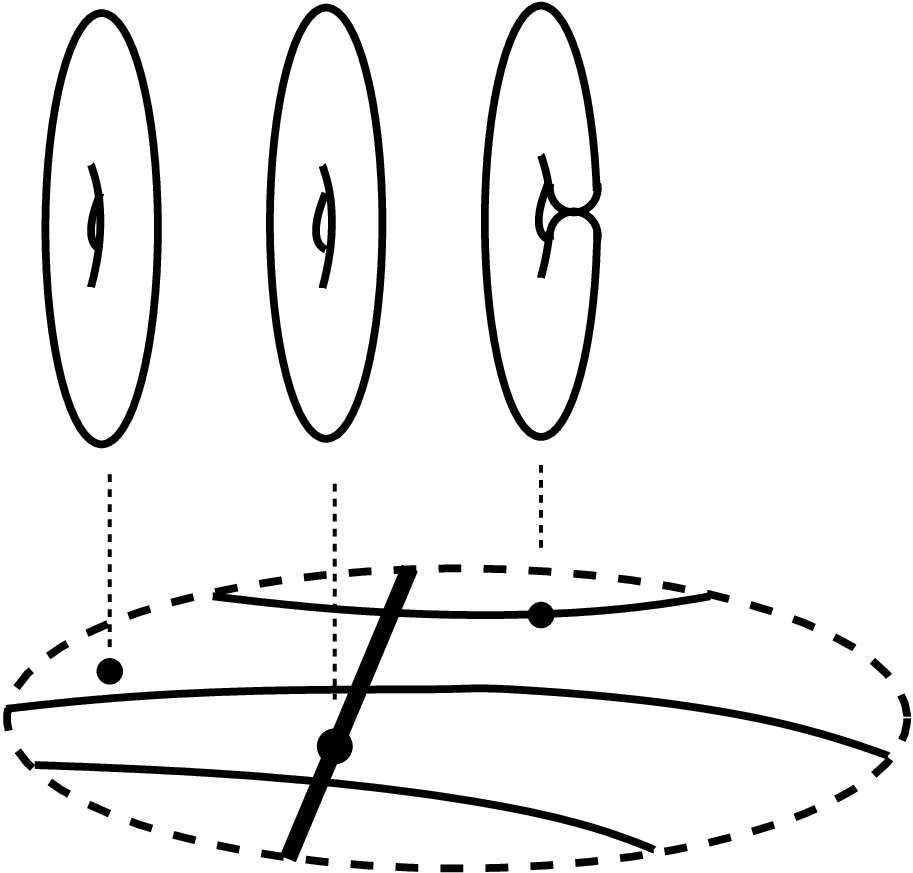}
\caption{The simultaneous stable reduction}
\end{figure}
By definition, we obtain a morphism
$$
\widetilde{B}'\to \overline {\mathcal M}_{1,1} \to \overline M_{1,1}. 
$$
The map $\widetilde B'\to \widetilde B$ is finite, so in fact there is a map $\widetilde B\to \overline M_{1,1}$ (e.g.~\cite[Lem.~2.4]{cautis}).  
The locus $\{c_1=0\}$ can be identified with the $\lambda$-line, with corresponding family given by $y^2=x(x-1)(x-\lambda)$.  The points $\lambda=0,1,\infty$ correspond to the intersections of the strict transforms of the hyperplane arrangement (the discriminant after the Weyl cover).  The restricted map $\{c_1=0\}\to \overline M_{1,1}$ can be identified with the map from the $\lambda$-line to the $j$-line.

\subsection{Obstructions}\label{secnegex}
We have seen that there exists an extension of the rational map $\tilde B\dashrightarrow  \overline{M}_{1,1}$ to the moduli scheme $\tilde B\to \overline{M}_{1,1}$.   We now show the rational map $\tilde B\dashrightarrow  \overline{\mathcal M}_{1,1}$ to the moduli stack does not extend; i.e. ~there is no family of stable curves over $\tilde B$ extending the pull-back of the original family.

We  do this in the following way.  Let $S$ be the spectrum of a DVR with closed point $s$ and generic point $\eta$.   We will find a morphism $S\to \tilde B$ sending $s$ to a closed point of the exceptional divisor (i.e.~$\{c_1=0\}$, parameterizing the cuspidal locus) and sending $\eta$ to the generic point of $\tilde B'$ (i.e.~the smooth locus).  Then we will show that the induced family of curves $X_S\to S$ does not extend to a family of stable curves; i.e.~the composition $S\to \tilde B\dashrightarrow \overline{\mathcal M}_{1,1}$ does not extend to a morphism.

We will show in two ways that the general  $S\to \tilde B\dashrightarrow \overline{\mathcal M}_{1,1}$ as above does not extend to a morphism.  The first is via a monodromy computation.  
   The second method is via a computation following an argument of  Fedorchuk \cite{fedorchuk}.

\subsubsection{The monodromy obstruction} \label{secmonobexa}
 Consider the family $X'\to B'$ obtained via the Weyl cover, and the restriction  $(X')_{\mid L}\to L$ of this family to a generic line $L$ through the origin in $B'$.
To show that there is no extension $\tilde B\to \overline{\mathcal M}_{1,1}$ to the moduli stack, it suffices to show that the restriction $(X')_{\mid L}\to L$ does not extend to a stable family of curves.  To show this,  observe that in the notation of \S \ref{secwbu}, the restriction $(X')_{\mid L}\to L$ is a surface $Z_{b_2}$ with equation (locally near the $A_2$ singularity):
\begin{equation}\label{fam3}
x_2^2+x_1^3-(b_2^2+b_2+1) b_1^2 x_1-b_2(1+b_2) b_1^3 =0,
\end{equation}
where $b_1$ is a parameter for $L$ and $b_2$ is a (generic) fixed slope.
    
  The surface $Z_{b_2}$ has a $D_4$ singularity at the origin.  This is also a cusp singularity for $X_0$, the central fiber of $Z_{b_2}$ viewed as a family of curves.  Recall that the standard resolution of a $D_4$ surface singularity $x^2=f_3(y,z)$ is given by $4$ blow-ups:  First blow-up the $D_4$ singularity.  This gives an exceptional divisor $E_0$.  The $D_4$ singularity ``splits'' into three $A_1$ singularities corresponding to the three roots of $f_3$. Then  blow-up  each $A_1$ singularity.  This  introduces  exceptional divisors $E_1,E_2,E_3$, giving the desired resolution.  We associate to this a  $\widetilde{D}_4$ graph (consisting of $E_0$ the central vertex, to which one attaches edges connecting the $4$ vertices corresponding to the curves $X_0$, $E_1$, $E_2$ and $E_3$).
 
The monodromy obstruction can be identified via the theory of elliptic fibrations. 
From the $\widetilde D_4$ graph, we conclude that this is a type $I_0^*$ degeneration in Kodaira's classification  (see \cite[\S V.7, p.201]{barth}). 
It follows that the monodromy  is  $-\operatorname{Id}$ (see \cite[p.210]{barth}).
We also direct the reader to the discussion of the elliptic involution in \cite[Ch. 2A]{hm}, and  to the monodromy computation made in \S \ref{secmonocusp}.  In conclusion, the monodromy not being unipotent, the family does not extend to a family of stable curves (this fact is reviewed in \S \ref{secab} and \S \ref{seccurves}).

\subsubsection{An obstruction via a direct computation} 
Again, our goal is to show 
 that the general map  $S\to \tilde B\dashrightarrow \overline{\mathcal M}_{1,1}$ with closed point sent to the cuspidal locus, and generic point sent to the smooth locus, does not extend to a morphism.

We follow an observation of Fedorchuk \cite[Prop.~7.4]{fedorchuk}.   Let  $S'$ be the  spectrum of  a DVR, which is a branched double cover of $S$ admitting a morphism  to $\tilde B'$.    Pulling back the family $\widehat{ X}\to \tilde B'$ we obtain a family of stable curves $\widehat { X}_{S'}\to S'$.   Fedorchuk's observation is that it suffices to show that the total space $\widehat { X}_{S'}$ is smooth.  Indeed,  if there were a family of stable curves $\widehat { X}_S\to S$ extending the pull-back of the original family, then it would follow that $\widehat { X}_{S'}$ was equal to $S'\times_S\widehat { X}_{S}$.   But then the total space $\widehat { X}_{S'}$ would have a singular point, giving a contradiction (the point of $\widehat { X}_{S}$ at the node of the central fiber, locally given by $xy-t^n$ with $n\ge 1$, would be replaced with a singular point $xy-t^{2n}$ of $\widehat{ X}_{S'}$).  

Fedorchuk's approach is to construct a particular one-parameter family of genus two curves degenerating to a cusp (his argument implies the result for families of curves of arbitrary genus  \cite[Prop.~7.4]{fedorchuk}).  
Alternatively, with the work we have done here in coordinates, one can show that the blow-up in the fourth step gives a smooth total space when restricted to  the general $S'$.  Since the total space is a smooth surface, and all of the curves blown down in the fifth step are  $(-1)$-curves, this  does not introduce singularities in the total space.


\section{Stable reduction and the valuative criterion for properness} \label{secsrstacks}

We now consider stable reduction more abstractly, in terms of the valuative criterion for properness for stacks.  For readers not familiar with stacks, this is not strictly necessary for the material in the subsequent sections.   The main focus here will be to review the fact that stable reduction is equivalent to the properness of a moduli stack, and that for separated, finite type  Deligne--Mumford stacks, the properness of the coarse moduli space is equivalent to the properness of the stack.
In order to return quickly to a more concrete setting, we postpone a discussion of simultaneous stable reduction in the language of stacks until \S \ref{secssrs}.

\subsection{The valuative criterion for properness of an algebraic stack}  Fix once and for all a scheme $Z$, and consider the \'etale site $\operatorname{Sch}_Z^{\text{\'et}}$ (e.g.~\cite[Exa.~2.3.1, p.27]{fgae}).   By a $Z$-sheaf, (resp.~algebraic $Z$-space, resp.~$Z$-stack), we will mean a sheaf (resp.~algebraic space, resp.~stack) on   $\operatorname{Sch}_Z^{\text{\'et}}$ (e.g.~\cite[Def.~1.1, Def.~3.1]{lmbchampes}).

An \textbf{algebraic (or Artin) $Z$-stack $\mathcal M$ over  $\operatorname{Sch}_Z^{\text{\'et}}$} is a $Z$-stack such that the diagonal $1$-morphism of $Z$-stacks $\Delta: \mathcal M\to \mathcal M \times_Z \mathcal M$ is representable, separated and quasi-compact, and there exists an algebraic $Z$-space $U$ and a $1$-morphism of $Z$-stacks $U\to \mathcal M$ that is surjective and smooth (e.g.~\cite[Def.~4.1]{lmbchampes}).      Note that the diagonal $\Delta$ is in fact of finite type (e.g.~\cite[Lem.~4.2]{lmbchampes}).     
An algebraic $Z$-stack $\mathcal M$ is a \textbf{Deligne--Mumford (DM) $Z$-stack} if there is a 
an algebraic $Z$-space $U'$ and a $1$-morphism of $Z$-stacks $U'\to \mathcal M$ that is surjective and  \'etale  (e.g.~\cite[Def.~4.1]{lmbchampes}).

\begin{exa}
For a concrete example we will consider  
$\overline{\mathcal M}_g$, the DM   $\mathbb C$-stack  of genus $g\ge 2$ curves over $\operatorname{Sch}_{\mathbb C}^{\text{\'et}}$.  This is a category  whose objects are families  $X\to B$ of Deligne--Mumford stable, genus $g$ curves, over a $\mathbb C$-scheme $B$ (see \S \ref{seccurves}), and whose morphisms are given by pull-back diagrams.  Recall that the moduli stack ``represents the moduli problem'' in the following way: for a $\mathbb C$-scheme $B$, a morphism $B\to \overline{\mathcal M}_g$ is equivalent to a family $X\to B$ of stable, genus $g$ curves, over $\mathbb C$.  
\end{exa}

We now state the valuative criteria for separateness and properness.   We refer the reader to \cite[Def.~7.6, Def.~7.11]{lmbchampes} (see also \cite[Def.~4.7, Def.~4.11]{dm}, \cite[Def.~1.1]{vistoli89}) for the respective definitions of separateness and properness, and we note only that separated and finite type are assumed in the definition of properness.

\begin{teo}[Valuative Criterion for Separatedness]   Let $F:\mathcal M\to \mathcal B$ be a $1$-morphism of algebraic $Z$-stacks.  Then $F$ is separated if and only if for every valuation ring $R$,  with field of fractions $K$, 
and every $2$-commutative diagram 
\begin{equation}\label{eqnvcs}
\xymatrix{
& \mathcal M \ar@{->}[d]^F\\
\operatorname{Spec} R \ar@{->}[r] \ar @{-->}^{x_1} @< 2pt> [ru] \ar@{-->}_{x_2} @<-2pt> [ru]& \mathcal B\\
}
\end{equation}
any isomorphism between $x_1|_{\operatorname{Spec}K}$ and $x_2|_{\operatorname{Spec}K}$ can be extended to an isomorphism between $x_1$ and $x_2$.    If moreover $\mathcal B$ is locally noetherian and $F$ is locally of finite type, then one need only consider discrete valuation rings $R$.
\end{teo}

We direct the reader to \cite[Prop.~7.8]{lmbchampes} (see also \cite[Thm.~4.18]{dm}).

\begin{rem}
The criterion, and in particular the $2$-commutivity,  can be made more explicit as follows. 
For all $x_1,x_2\in \operatorname{ob}\mathcal M_{\operatorname{Spec}R}$, all isomorphsims $\beta:F(x_1)\to F(x_2)$ in $\mathcal B_{\operatorname{Spec}R}$,   and all isomorphisms $\alpha:(x_1)|_{\operatorname{Spec}K}\to (x_1)|_{\operatorname{Spec}K}$ in $\mathcal M_{\operatorname{Spec}K}$ such that $F(\alpha)=\beta|_{\operatorname{Spec}K}$, there exists at least one (and in fact a single) isomorphism $\tilde \alpha:x_1\to x_2$ in $\mathcal M_{\operatorname{Spec}R}$ extending $\alpha$ (i.e.~$\tilde \alpha|_{\operatorname{Spec}K}=\alpha$) and such that $F(\tilde \alpha)=\beta$.   We also note that it suffices to take valuation rings that are complete, and have algebraically closed residue field (see \cite[Prop.~7.8]{lmbchampes}).
\end{rem}

\begin{teo} [Valuative Criterion for Properness]  
  Let $F:\mathcal M\to \mathcal B$ be a separated, finite type $1$-morphism of algebraic $Z$-stacks.  Then $F$ is proper if and only if for every discrete valuation ring $R$,  with field of fractions $K$, 
and every $2$-commutative diagram 
$$
\xymatrix{
\operatorname{Spec}K  \ar@{->}[r]  \ar@{->}[d]& \mathcal M \ar@{->}[d]^F\\
\operatorname{Spec} R \ar@{->}[r] & \mathcal B\\
}
$$
there exists a finite extension $K'$ of $K$, so that taking $R'$ to be the integral closure of $R$ in $K'$, there is  a $2$-commutative diagram
\begin{equation}\label{eqnvcp}
\xymatrix{
\operatorname{Spec}K' \ar@{->}[r]  \ar@{->}[d]& \operatorname{Spec}K  \ar@{->}[r]  \ar@{->}[d]& \mathcal M \ar@{->}[d]^F\\
\operatorname{Spec}R'  \ar@{->}[r]  \ar@{-->}[rru]& 
\operatorname{Spec} R \ar@{->}[r] & \mathcal B\\
}
\end{equation}
extending the original diagram.
\end{teo}

We direct the reader to \cite[Thm.~7.10]{lmbchampes} (see also \cite[Thm.~4.19]{dm}).

\begin{rem}
It suffices to consider DVRs that are complete, and have algebraically closed residue field (see \cite[Thm.~7.10]{lmbchampes}).  If one removes the hypothesis that $F$ be separated, then the criterion \eqref{eqnvcp} is equivalent to $F$ being universally closed (e.g.~\cite[Thm~7.10]{lmbchampes}).
\end{rem}

\begin{exa}
Let us make these criteria more concrete in  the example of $\overline{\mathcal M}_g$, $g\ge 2$.   We will use the fact that $\overline {\mathcal M}_g$ is of finite type over $\mathbb C$ \cite[Thm.~5.2]{dm} together with the fact that $\mathbb C$ is noetherian to conclude that we need only  consider DVRs.  Then  using the fact that a morphism from a scheme to $\overline{\mathcal M}_g$  is the same as family of stable curves over the scheme, we may  reinterpret the valuative criteria as follows.  

\emph{Separatedness}:  Let $R$ be a DVR with fraction field $K$.  Set $S=\operatorname{Spec} R$.  Suppose $X\to S$ (resp.~$Y\to S$) is  a family of stable, genus $g$ curves over $S$, with restriction $X_K\to \operatorname{Spec}K$ (resp.~$Y_K\to \operatorname{Spec}K$).  Then the valuative criterion for separateness requires that any  $K$-isomorphism $X_K\to Y_K$ extend to an $S$-isomorphism of $X/S\to Y/S$.  

\emph{Properness}:  Let $R$ be a DVR with fraction field $K$.  Let  $X_K\to \operatorname{Spec}K$ be a family of stable, genus $g$ curves.  Then the  valuative criterion for properness requires that  there  exist a finite extension $K'$ of $K$ such that setting $R'$ to be the integral closure of $R$ in $K'$, there exists a family of stable curves  $X'\to \operatorname{Spec}R'$ extending the family $X_{K'}=X_K\times_{\operatorname{Spec}K} \operatorname{Spec}K'\to \operatorname{Spec}K'$ (in the sense that $X'|_{K'}\cong X_{K'}$).

In conclusion, the properness of $\overline{\mathcal M}_g$ is equivalent to the fact that any family of stable curves over the generic point of a DVR can be extended, possibly after a generically finite base change, to a family of stable curves, and this extension is unique up to isomorphism.   This is exactly the statement of the Deligne--Mumford stable reduction theorem, which we will review in \S \ref{seccurves}.  
\end{exa}

\begin{rem}    In practice,  a natural moduli problem (over a scheme $Z$) will often lead to a separated, non-proper,  algebraic $Z$-stack $\mathcal M$ of finite type over $Z$.   
A stable reduction theorem for the moduli problem then consists of finding a proper algebraic $Z$-stack $\overline {\mathcal M}$, containing $\mathcal M$ (ideally as a dense  open substack).  In general, finding such stacks has proven to be quite difficult.  One approach is to use GIT to determine a (GIT-)stability condition, with the hope that the stable objects will provide the correct class to define $\overline{\mathcal M}$; we discuss this further in \S \ref{secgit}.    Another approach, due to Koll\'ar--Shepherd-Barron--Alexeev, which uses the Minimal Model Program (MMP), has had a great deal of success lately;  we discuss this  further in \S\ref{secsrhd}.
\end{rem}

\begin{rem}
 It may also happen that there  are many such proper stacks $\overline {\mathcal M}$ from which to choose.
Following Smyth (e.g.~\cite{smyth09}), who has investigated this  question extensively for the moduli of curves,   
a useful way to frame the problem is as follows.
By considering all ``degenerations'' of objects in $\mathcal M$, one may obtain a ``highly non-separated'' algebraic $Z$-stack $\mathcal U$, which contains $\mathcal M$ as an open substack.  Essentially by construction, the stack $\mathcal U$ should satisfy the valuative criterion \eqref{eqnvcp}.   One is then in the situation of identifying proper substacks $\overline {\mathcal M}$ of $\mathcal U$ that contain $\mathcal M$.  We direct the reader to Smyth \cite{smyth09} for more on this, especially for the case of curves 
(see also \cite{asv10},\cite{as12} where a notion weaker than properness is considered).
\end{rem}

\subsection{Moduli spaces}
A moduli space for a stack is an algebraic space  or scheme that is as close as possible to the stack.  More precisely, a \textbf{categorical moduli space} for an algebraic $Z$-stack $\mathcal M$ is a $Z$-morphism $\pi:\mathcal M\to M$ to an algebraic $Z$-space such that $\pi$ is initial for $Z$-morphisms to algebraic $Z$-spaces.  This means that  
 given any $Z$-morphism $\Phi:\mathcal M\to Y$ to an algebraic $Z$-space, there is a unique $Z$-morphism $\eta:M \to Y$ making the following diagram commute
$$
\xymatrix{
\mathcal M \ar@{->}[r]^\Phi \ar@{->}[d]_\pi& Y\\
M \ar@{-->}[ru]_{\exists ! \eta} & \\
}
$$
We will call $M$ a \textbf{categorical moduli scheme} if $M$ is a $Z$-scheme.

A \textbf{coarse moduli space (resp.~scheme)}  is a categorical moduli space (resp. scheme) satisfying the additional condition that for every algebraically closed field $k$, the induced map $|\mathcal M(k)|\to M(k)$ is a bijection, where $|\mathcal M(k)|$ is the set of isomorphism classes of the groupoid $\mathcal M_{\operatorname{Spec}k}$. 
For stacks with finite inertia there is the following theorem of Keel--Mori.   
Recall that for an algebraic $Z$-stack $\mathcal M$, the inertia stack $I_Z(\mathcal M)$ is the fiber product $\mathcal M\times _{\mathcal M\times_Z\mathcal M}\mathcal M$, where both morphisms $\mathcal M\to \mathcal M\times_Z\mathcal M$ are the diagonal.   An algebraic $Z$-stack is said to have finite inertia if $I_Z(\mathcal M)$ is finite over $\mathcal M$.   Note that by pull-back, a stack with finite diagonal, and hence a separated, finite type  DM stack (e.g. \cite[Lem.~1.13]{vistoli89}),  has finite inertia.  

\begin{teo}[Keel--Mori {\cite[Cor.~1.3]{keelmori97}}]  Let $\mathcal M$ be an algebraic $Z$-stack, locally of finite presentation, with finite innertia. Then there exists a coarse moduli space $\pi:\mathcal M\to M$, with $\pi$ proper.  If  $\mathcal M/Z$ is separated, then $M/Z$  is separated.  If $Z$ is locally notherian, then $M/Z$ is locally of finite type.  If $Z$ is locally noetherian and $\mathcal M/Z$ is of finite type with finite diagonal, then $\mathcal M/Z$ is proper if and only if $M/Z$ is proper.  
\end{teo}

For a proof, we direct the reader to Keel--Mori \cite{keelmori97} (see also Conrad \cite[Thm.~1.1]{conradKM} and Olsson \cite[Rem.~1.4.4]{olsson}).  
We also direct the reader to the definition of a tame stack in  Abramovich--Olsson--Vistoli \cite[Def.~3.1]{AOV08}.

\begin{rem}\label{remprojmodspace}
As a consequence of the theorem, one can prove a stable reduction theorem for a moduli problem (with a reasonable moduli stack) by showing that the moduli stack admits a proper moduli space.  One standard approach to constructing a proper moduli space is via GIT (\S \ref{secgit}), where one will in fact typically obtain the stronger  statement that the moduli space is projective.   
Note that alternatively, for a proper moduli space, one can use positivity results of Koll\'ar \cite{kollar90} to establish the projectivity of the moduli space directly.
\end{rem}

\begin{exa}
As $\overline{\mathcal M}_g$ is a proper DM $\mathbb C$-stack, the Keel--Mori theorem implies there is  proper coarse moduli space $\overline{\mathcal M}_g\to \overline M_g$.   In fact, the moduli space is a projective variety over $\mathbb C$.  In \S \ref{secgit}, we will sketch a GIT construction of $\overline M_g$ due to Gieseker \cite{gieseker}.  By the  valuative criteria, the existence of a projective coarse moduli space provides another proof of stable reduction.   Note also, that as an application of the techniques in \cite{kollar90}, Koll\'ar gives an independent proof that $\overline M_g$  is projective (\cite[Thm.~5.1]{kollar90}).
\end{exa}

For more on algebraic stacks with positive dimensional stabilizers, the reader is directed to Alper \cite{alpergms}.  See especially the definition \cite[Def.~4.1]{alpergms} of a good moduli space.  We point out that (under mild hypotheses) a good moduli space is a categorical moduli space (\cite[Thm.~6.6, Thm.~4.16(vi)]{alpergms}).  If  $\pi:\mathcal M\to M$ is a good moduli space, then $\pi$ may fail to be separated, but it does satisfy the valuative criterion \eqref{eqnvcp}; i.e.~it is universally closed (\cite[Thm.~4.16(ii)]{alpergms}, see also \cite[Prop.~2.17]{asv10}).  
An illustrative example is the morphism  $\pi:[\mathbb A_k^1/\mathbb G_m]\to \operatorname{Spec}k$, for an algebraically closed field $k$ (see e.g.~\cite[Exa.~8.6]{alpergms}).    This is a good (categorical) moduli space, which is not coarse, and such that $\pi$ is universally closed, but not separated.


\section{Semi-stable reduction} \label{secssr}

Semi-stable reduction is the process of filling in the central fiber of a family of smooth varieties over the punctured disk with  a reduced scheme with normal crossing singularities.  This is the natural generalization of filling in  the central fiber in a family of smooth curves with a nodal curve.   A semi-stable reduction provides a simple special fiber, which is useful from many points of view, such as Hodge theory, period maps, and monodromy.   On the other hand, unlike a stable reduction, a semi-stable reduction is not unique.  

In this section we discuss a theorem of Mumford et al.~\cite{mumetal} that establishes the existence of semi-stable reductions in characteristic zero.  In the next section, we will discuss the connection with monodromy, which plays a central role in the stable reduction theorem for abelian varieties.  In \S \ref{seckol} we will use the semi-stable reduction theorem in discussing an approach of Koll\'ar--Shepherd-Barron--Alexeev  to establishing stable reduction theorems.

\subsection{Semi-stable Reduction Theorem}

We begin by stating the semi-stable reduction theorem of  Kempf--Knudsen--Mumford--Saint-Donat \cite{mumetal}.

\begin{teo}[Semi-stable Reduction Theorem {\cite[Thm.~p.53]{mumetal}}]\label{teossr}   Assume \\ that $\operatorname{char}(k)=0$ and $k=\overline k$.   Let $B$ be on open subset of a non-singular curve over $k$, fix a point $o\in B$, and set $U=B-\{o\}$.  Suppose that 
$$
\pi:X\to B
$$
is a surjective morphism of a variety $X$ onto $B$ such that the restriction 
$
\pi_U:X|_U\to U
$
is smooth.  Then there is a finite base change $f:B'\to B$, with $B'$ non-singular and $f^{-1}(o)$ a single point $o'$, a non-singular variety $X'$ and a diagram
\begin{equation}\label{eqnsrdiag}
\xymatrix
{
 X'\ar@{->}[r]_{p \ \ \ } \ar@{->}[rd]_{\pi'} \ar@{->}@/^1pc/[rr]& B'\times_B X \ar@{->}[r]\ar@{->}[d]& X\ar@{->}[d]^\pi\\
 & B'\ar@{->}[r]^f & B 
}
\end{equation}
satisfying the properties below.
\begin{enumerate}
\item Setting $U'=B'-\{o'\}$, $p$ is an isomorphism over $U'$.
\item $(\pi')^{-1}(o')$ is a reduced scheme, which is an snc divisor on $X'$.
\item The morphism $p$ is projective, and given as a blow-up of an ideal sheaf $\mathcal I$ that is trivial away from the fiber over $o'$.
\end{enumerate}
\end{teo}

This result is used so frequently in stable reduction arguments, and parts of the  proof are so constructive, that it is worthwhile to sketch the outline here.  One of the key points is the following example.

\begin{exa}
Consider the variety $X$ in $\operatorname{Spec}k[\underline x,t]=\mathbb A^{r+1}_k$ defined by 
$$
t-x_1^{a_1}\cdots x_r^{a_r}.
$$
We view $X$ as a family $\pi:X\to B:=\operatorname{Spec}k[t]$, with central fiber $D=\pi^{-1}(0)$.     For each $d\in  \mathbb N$, set $B_d=\operatorname{Spec}k[t]$, and $f_d:B_d\to B$ to be the map given by $t\mapsto t^d$.  
We define $X_d$ to be the normalization of the pull-back of $X$ via the map $B_d\to B$.  In other words, we have a diagram
$$
\xymatrix
{
 X_d\ar@{->}[r]_{\nu_d \ \ \ } \ar@{->}[rd]_{\pi_d} \ar@{->}@/^1pc/[rr]& B_d\times_B X \ar@{->}[r]\ar@{->}[d]& X\ar@{->}[d]\\
 & B_d\ar@{->}[r]^{f_d} & B 
}
$$
\emph{In this example, we will assume that  
$$
d=\operatorname{lcm}(a_1,\ldots,a_r) \ \ \text{ and } \ \ \operatorname{gcd}(d,a_1,\ldots,a_r)=\operatorname{gcd}(a_1,\ldots,a_r)=1,
$$
and we will describe $X_d$ and $\pi_d^{-1}(0)$}.  
First,  $X_{B_d}:=B_d\times_B X$  is defined by
$$
t^d-x_1^{a_1}\cdots x_r^{a_r}.
$$
The assumption 
$\operatorname{gcd}(d,a_1,\ldots,a_r)=\operatorname{gcd}(a_1,\ldots,a_r)=1$ implies that $X_{B_d}$ is the image of the morphism
\begin{equation}\label{eqnmummap}
\operatorname{Spec}k[\underline y]=\mathbb A^r_k\to \mathbb A^{r+1}_k=\operatorname{Spec}k[\underline x,t]
\end{equation}
given by $
(y_1,\ldots,y_r)\mapsto (y_1^d,\ldots,y_r^d,y_1^{a_1}\cdots y_r^{a_r})$.  The  map  \eqref{eqnmummap} factors as 
$$
\mathbb A^r_k\to \mathbb A^r_k=\operatorname{Spec}k[\underline z]\to \mathbb A^{r+1}_k
$$
where the first map is given by $
(y_1,\ldots,y_r)\mapsto (y_1^{a_1},\ldots,y_r^{a_r})$ and the second map is given by $(z_1,\ldots,z_r)\mapsto (z_1^{d/a_1},\ldots,z_r^{d/a_r},z_1\cdots z_r)$.   In short we have 
$$
\operatorname{Spec}k[\underline y]\to \operatorname{Spec}k[\underline z]\to X_{B_d}= \operatorname{Spec}\left(k[\underline x,t]/(t^d-x_1^{a_1}\cdots x_r^{a_r}) \right)
$$
and the associated morphisms of rings are the inclusions:
\begin{equation}\label{eqnre}
k[y_1,\ldots,y_r]\supseteq k[y_1^{a_1},\ldots,y_r^{a_r}] \supseteq k[y_1^d,\ldots,y_r^d, y_1^{a_1}\cdots y_r^{a_r}].
\end{equation}

Let us consider for a moment  the special case where  $\operatorname{Spec}k[\underline z]\to X_{B_d}$ is birational.   This will be the case, for instance,  if either $r=2$, or, more generally, if $a_3=\cdots=a_{r}=a_1a_2$ (but will fail in general; e.g.~the case $X=V(t-x_1^2x_2x_3)$).  
Then it follows from Zariski's main theorem that $$\operatorname{Spec}k[\underline z]\to X_{B_d}$$  is  the  normalization $\nu_d:X_d\to X_{B_d}$. 
The divisor $D$ corresponds to $(t)$ in $k[\underline x,t]/(t^d-x_1^{a_1}\cdots x_r^{a_r})$, which corresponds to $z_1\cdots z_r$ in $k[\underline z]$.   In  conclusion, in this special case, $X_d$ is smooth and $\pi_d^{-1}(0)$ is a reduced, nc divisor.
 
In general, describing the normalization $\nu_d:X_d\to X_{B_d}$ is more complicated.  From the ring on the right in \eqref{eqnre}, one readily obtains a toric description of $X_{B_d}$.  The normalization can then be described in terms of associated semi-groups (see \cite[p.101]{mumetal}).  Using this approach, it is established in \cite[Lem.~1, p.102, Lem.~2, p.103]{mumetal} that $\pi_d^{-1}(0)$ is reduced, and the pair $X_d$ and $\pi^{-1}(0)$ give rise to a toroidal embedding without self-intersection.    We discuss toroidal embeddings briefly in \S  \ref{secak}.  In the case where $X_d$ is non-singular, we point out that this implies that $\pi_d^{-1}(0)$ is nc.
\end{exa}

\begin{exa}
Again consider the variety $X$ in $\operatorname{Spec}k[\underline x,t]$ defined by 
$
t-x_1^{a_1}\cdots x_r^{a_r}
$.  
Using the same notation as in the previous example, \emph{we will keep the assumption that $d=\operatorname{lcm}(a_1,\ldots,a_r)$, but will discard the assumption that $\operatorname{gcd}(d,a_1,\ldots,a_r)=\operatorname{gcd}(a_1,\ldots,a_r)=1$}.
The fibered product  $X_{B_d}:=B_d\times_B X$ is defined by
$
t^d-x_1^{a_1}\cdots x_r^{a_r}
$.  
Setting $e=\operatorname{gcd}(d,a_1,\ldots,a_r)$, this family decomposes as
$
\prod_{\zeta^e=1}\left(t^{d/e}-\zeta \prod_{i=1}^rx_1^{a_1/e}\cdots x_r^{a_r/e}\right)
$.
Since $d/e=\operatorname{lcm}(a_1/e,\ldots,a_r/e)$, and $\operatorname{gcd}(d/e,a_1/e,\ldots,a_r/e)=1$, we see that we can reduce to the case of  ($e$ copies of) the previous example.
\end{exa}

\begin{exa}
Again consider the variety $X$ in $\operatorname{Spec}k[\underline x,t]$ defined by 
$
t-x_1^{a_1}\cdots x_r^{a_r}
$.
Using the same notation as above, set $\ell=\operatorname{lcm}(a_1,\ldots,a_r)$,   \emph{assume that $d=n\cdot \ell$ for some $n\in \mathbb N$, and again discard the assumption that $\operatorname{gcd}(d,a_1,\ldots,a_r)=\operatorname{gcd}(a_1,\ldots,a_r)=1$}.
   One can show (see e.g.~\cite[Lemma 2, p.103]{mumetal}) that $X_d=B_d\times_{B_\ell} X_{\ell}$. 
 \end{exa}

\begin{rem}\label{remssrexa}
In summary, for the variety $X\subseteq \operatorname{Spec}k[\underline x,t]$ defined by 
$
t-x_1^{a_1}\cdots x_r^{a_r},
$ 
in the notation above if $d=n \cdot \operatorname{lcm}(a_1,\ldots,a_r)$ for some $n\in \mathbb N$, then $X_d$ consists of $e=\operatorname{gcd}(a_1,\ldots,a_r)$ connected components.  On each of these components, $\pi_d^{-1}(0)$ gives rise to a toroidal embedding without self-intersection.    
For surfaces, the singularities appearing on $X_d$ will be at worst of type $A$ (the definition of a type $A$ singularity is recalled in  \S \ref{secsing}).  
\end{rem}

We now briefly outline the Mumford et al.~ proof of the Semi-stable Reduction Theorem (see \cite[pp.98-108]{mumetal} for more details).

\begin{proof}[Sketch of the proof of the Semi-stable Reduction Theorem]  \ \ \ \ 
Let $\pi:X\to B$ be a morphism as in the statement of the theorem.  Using the  characteristic  zero assumption, perform a log resolution  of the pair $(X,\pi^{-1}(o))$.  We obtain a new family $\tilde \pi:\tilde X\to B$, where $\tilde \pi^{-1}(o)$ is normal crossing (although it may not be reduced).      Setting $\ell$ to be the $\operatorname{lcm}$ of the multplicities of the components of $\tilde \pi^{-1}(o)$,  make a base change of degree $\ell$, and then normalize the total space.   

Call the space obtained $X^\nu \to B_\ell$.
The claim is that $X^\nu$ satisfies the conditions of the theorem, with the possible exception that $X^\nu$ may fail to be smooth (in which case the central fiber may only induce a toroidal embedding without self intersection, rather than being nc).
Indeed, the question is \'etale local, so to describe $X^\nu$ we may  reduce to the examples we have already considered, where $\tilde X$ is defined by 
$$
t-\prod_{i=1}^rx_1^{a_1}\cdots x_r^{a_r}.$$

The remaining issue is to resolve the singularities of the total space of $X^\nu$ (while retaining the property that the central fiber induces a toroidal embedding without self-intersection).  This is done in \cite[pp.104-108]{mumetal}; note that a further base change may be required (see \cite[p.107]{mumetal}).
\end{proof}

\begin{rem}
For the case of families of curves (where $X$ is a surface), the total space of $X^\nu$ has type $A$ singularities, and one can achieve the resolution in the final step above by a sequence of blow-ups introducing chains of rational curves.
\end{rem}

\section{Monodromy}\label{secmono}

The monodromy representation is a topological invariant associated to a family over a punctured disk.
It is the essential invariant in the context of period maps (see Remark \ref{remfinmon}), as well as for abelian varieties.
 In this section, we briefly review the definition of monodromy, and then compute a few examples.  We then state the monodromy theorem. 
While there is an algebraic monodromy representation for families over DVRs,  for simplicity, we restrict to the case of  monodromy in the analytic setting.

\subsection{Preliminaries on  monodromy}

Let $ X^\circ\to S^\circ$ be a smooth family of complex, projective varieties over the punctured disk.  
It is well known that for each $t_1,t_2\in S^\circ$, the fiber $X_{t_1}$ is diffeomorphic to the fiber $X_{t_2}$  (see e.g.~\cite[Thm.~2.3, p.61]{kodaira}).  
 In particular, the fibers are all homeomorphic, and the  cohomology groups $H^\bullet (X_t,\mathbb C)$ are isomorphic for all $t\in S^\circ$.    Fix a base point $\ast\in S^\circ$ and consider  a path $\gamma:[0,1]\to S^\circ$  that generates $\pi_1(S^\circ,\ast)$.    The family of groups $H^\bullet(X_{\gamma(\tau)},\mathbb C)$,  $\tau\in [0,1]$, determines an automorphism of $H^\bullet (X_\ast,\mathbb C)$.    The induced homomorphism
$$
\pi_1(S^\circ,\ast)\to \operatorname{Aut}H^\bullet(X_\ast,\mathbb C)
$$
is called the (analytic) monodromy representation of the family.

\begin{figure}[htp]\label{figmon}
\includegraphics[scale=0.25]{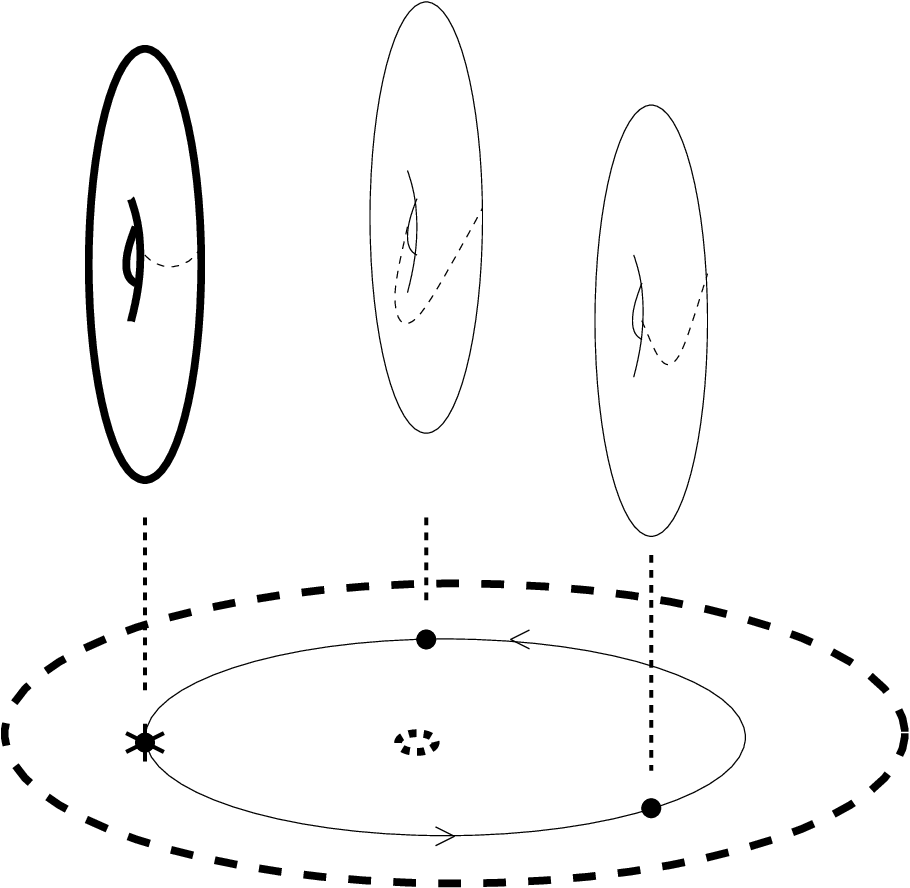}
\caption{Monodromy}
\end{figure}

For a family that extends to a smooth family over $S$, the monodromy representation  is trivial. 
In this sense, monodromy is an invariant that is meant to detect something about the  singularities of the central fiber of a degeneration.   For instance, ordinary double points ($A_1$ singularities)  give rise to the so called Picard--Lefschetz transformations (see the example in \S \ref{secmononode}).  
On the other hand, we note that it is not the case that trivial monodromy implies that a family can be extended to a smooth family over the disk.  
For an elementary example, see Remark \ref{remtrivmon}.  For a more interesting example, see Friedman \cite{friedman83}.

Recall that an endomorphism $T$ of a finite-dimensional vector space $V$ is said to be \textbf{unipotent} (resp. \textbf{quasi-unipotent})  if there exist $M\ge 1$ (resp. $M,N\ge 1$) such that  $(T-\operatorname{Id}_V)^M=0$ (resp. $(T^N-\operatorname{Id}_V)^M=0$).

\subsubsection{A family of stable curves} \label{secmononode}
Consider the family
$$
x_2^2-(x_1^2-t)(x_1-1);
$$
i.e.~a family of smooth elliptic curves degenerating to a nodal cubic.  
Set $\ast=1/2$ and let $\gamma:[0,1]\to S^\circ$  be a parameterization of the circle of radius $1/2$.    The family of varieties lying over $\gamma$ is a family of elliptic curves determined by the branch locus $\{-\sqrt t, \sqrt t, 1,\infty\}$.   There is  a  basis    of $H^1( X_\ast,\mathbb C)$   for which the monodromy representation  is given by
$$
M_{A_1}=\left(
\begin{array}{cc}
1& 1\\
0& 1\\
\end{array}
\right).
$$
See Carlson--M\"uler-Stach--Peters \cite[p.18]{cmp} for a detailed exposition of this.  Note that this matrix is unipotent.  This transformation is in fact  a special case of   the Picard--Lefschetz theorem describing the monodromy transformations for degenerations to $A_1$ singularities (see e.g.~\cite[Ch.2, \S1.5]{arnoldsum}).

\subsubsection{A family of cuspidal curves}\label{secmonocusp}
Consider the family
$$
x_2^2-x_1^3-t;
$$
i.e.~a family of smooth elliptic curves degenerating to a cuspidal cubic.   Again, set $\ast=1/2$ and let $\gamma:[0,1]\to S^\circ$  be a paramaterization of the  circle of radius $1/2$.
There is a basis of  $H^1( X_\ast,\mathbb C)$   for which the monodromy representation  is given by
$$
M_{A_2}=\left(
\begin{array}{cc}
0& -1\\
1& 1\\
\end{array}
\right).
$$
\begin{figure}[htp]
\includegraphics[scale=0.25]{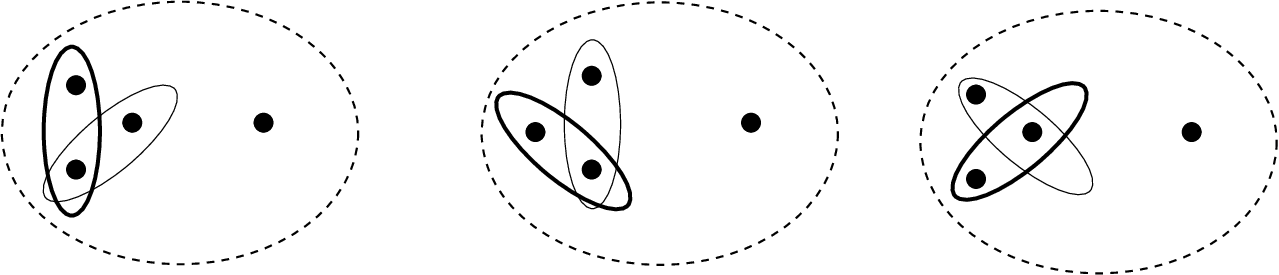}
\caption{Monodromy for a cuspidal family}\label{figmonA2}
\end{figure}
Figure \ref{figmonA2} shows the transformation of cycles on the copy of $\mathbb P^1$ lying below the elliptic curve,  with respect to the branch locus, for $t=\gamma(0)=1/2$, $t=\gamma(1/2)=-1/2$ and $t=\gamma(1)=1/2$, respectively.  Considering lifts of these cycles, one arrives at the matrix above.

In this example, the monodromy representation is quasi-unipotent, but not unipotent.
 Note additionally that $M_{A_2}^3=-\operatorname{Id}$.  This implies that after pulling the family back by a triple cover, the monodromy will be given by $-\operatorname{Id}$.    This new family gives a special case of the monodromy obstruction computation  made in \S \ref{secmonobexa}; in the notation of that section, this is the family given by taking $t_2=0$.  
 
Note further that since $M_{A_2}^6=\operatorname{Id}$,  if we pull the family back by a six-fold cover, the monodromy becomes trivial.  
This can be seen directly in the following way.  The family  obtained after a six-fold cover is 
$$
x_2^2-x_1^3-t^6.
$$
Changing coordinates by $x_1\mapsto x_1t^2$ and $x_2\mapsto x_2t^3$ gives the family
$$
x_2^2-x_1^3-1.
$$
In other words, after the degree six base change, the family can be extended to a \emph{trivial} family over $S$.  (Note the $j$-invariant of the original family was equal to zero for all $t\in S^\circ$.)

\begin{rem}  More generally, Kodaira has classified the degenerations of elliptic curves,   and their associated monodromy representations.   We direct the reader to 
\cite[\S V.7]{barth} for more details.
\end{rem}

\subsection{The monodromy  theorem}

\ \ The monodromy theorem is a general statement about the monodromy representation of a family of projective manifolds over the punctured disk.  This will play an important role in regards to an extension theorem of Grothedieck for abelian varieties, which  we discuss in  \S \ref{secab}.

\begin{teo}[Monodromy Theorem] Let $\pi^\circ:X^\circ \to S^\circ$ be a family of smooth, complex, projective manifolds of dimension $n$ over the punctured disk.  For each integer $0\le k\le 2n$,   the monodromy representation
$$
\pi_1(S^\circ,\ast)\to \operatorname{Aut}H^k(X_\ast,\mathbb C)
$$
is quasi-unipotent.
\end{teo}

 The reader is directed to Griffiths  
 \cite[Rem.~3.2, p.236]{griffiths70} for references, including a discussion of the history of the theorem and a description of the many different methods of proof (see also Grothendieck  \cite[Thm.~1.2, p.6]{sga7I} for the algebraic statement).

\begin{rem}\label{remqutou}
The monodromy theorem implies that for a family of smooth, complex, projective manifolds   over the punctured disk, after a finite base change   the monodromy  can be made unipotent.  Indeed, if the generator of the monodromy representation is given by the automorphism $T$, then   $(T^N-\operatorname{Id})^M=0$ for some $N,M$.  Thus after the base change given by $t\mapsto t^N$, the monodromy will be  unipotent.  We note that many of the proofs of the monodromy theorem provide bounds on $N$ and $M$. 
\end{rem}

\begin{rem}
If $\pi:X\to S$ is a family of complex projective varieties over the disk, which is smooth over the punctured disk,  and is such that $X_0:=\pi^{-1}(0)$ is a reduced snc divisor in $X$, then the monodromy representation is unipotent (see the references in \cite{griffiths70}).    One can deduce the Monodromy Theorem from this using the Semi-Stable Reduction Theorem \cite{mumetal} (Theorem \ref{teossr}).
\end{rem}

\begin{rem}
As is evident in the previous  remark, if $\pi:X\to S$ is a family of complex projective varieties over the disk,  which is smooth over the punctured disk, then the topology of $X_0$ is related to the monodromy of the family.    The Clemens--Schmid exact sequence makes this precise (e.g.~Morrison \cite[p.109]{morrison}).  There is also a notion of vanishing cohomology for isolated singularities on $X_0$.   There is a monodromy operator on the vanishing cohomology, which is related to the monodromy of the family by an exact sequence.  We direct the reader to \cite[p.V, 7--9]{sga7I} and \cite[(1.4)]{steenbrink} for more details.
\end{rem}

\begin{rem} \label{remfinmon}
Monodromy is the essential invariant in the context of period maps.   In particular, given a family $X^\circ \to S^\circ$ of smooth, projective varieties over the punctured disk, then via  Hodge theory one obtains a period map $S^\circ\to D/\Gamma$, where $D$ is the period domain and $\Gamma$ is an arithmetic group.   The period map extends to a morphism $S\to D/\Gamma$ if and only if the monodromy representation is finite; i.e.~generated by a root of the identity (e.g.~\cite[Thm.~13.4.5, p.355]{cmp}).
\end{rem}

\section{Abelian varieties}  \label{secab}

In this section we return to the question of stable reduction, and consider the case of abelian varieties.  Historically, this was one of the first places where questions about stable reduction were considered.  The monodromy theorem discussed in the previous section plays a central role, essentially due to the equivalence of categories between abelian varieties and Hodge structures of weight $1$.

Further motivation  comes from the  connection with the  original proof of the stable reduction theorem for curves in positive characteristic,
  which we discuss further in the next section. In this section we also consider stable reduction in the context of Alexeev's moduli space of stable semiabelic pairs \cite{alexeev}.

\subsection{An example of stable reduction for a family of abelian varieties}  
   
The family described in \S \ref{secexa}, viewed as a family of abelian varieties, gives a concrete example of stable reduction.  
Historically, however, one of the motivations for the development of the theory was to study abelian schemes over $\mathbb Q$  by reducing modulo a prime $p$.    Viewing the abelian scheme over $\mathbb Q$ as a family over the generic point of $\operatorname{Spec} \mathbb Z$, problems concerning reduction modulo a prime can be translated into problems about extending abelian schemes over $\mathbb Q$ to schemes over  $\operatorname{Spec}\mathbb Z$.

While a well known  theorem of Fontaine 
\cite[Cor., p.517]{fontaine} 
states there are no abelian schemes over $\operatorname{Spec}\mathbb Z$, so there can not  be an extension to an abelian scheme over every prime, the stable reduction theorem addresses the question of extending over a particular prime after a finite base change.  

With this as motivation, we consider the following example, which
 is closely related to the previous geometric examples, and  emphasizes  the connection between  the two  settings.
 We will use the terminology of group schemes, which we review in the next subsection.   

Let  $X\to \operatorname{Spec}\mathbb Z$ be the projective scheme defined by 
$$
y^2-x^3-25\alpha x-125\beta=0,
$$
with $\alpha$ and $\beta$ integers such that  $4\alpha^3+27 \beta^2$ is not divisible by $5$.   Let $X_{\mathbb Q}$ be the scheme obtained by base change to $\operatorname{Spec}\mathbb Q$.   There is the usual group law on $X_{\mathbb Q}$ induced by the point at infinity $(0:1:0)$.   
We are interested in understanding how this fails to extend to a group law on $X$ over $\operatorname{Spec}\mathbb Z$, and how one might attempt to rectify this at a particular prime by taking a finite cover.

Concerning the group law, one can check directly that  $X\to \operatorname{Spec}\mathbb Z$ fails to be smooth over (at least) the  primes $2$, $3$ and $5$, so that  $X_{\mathbb Q}$ does not extend to a group scheme over those points.    
In this example, we  focus on the issue of extension over $(5)$.  
Let $$X_{(5)}\to \operatorname{Spec}\mathbb Z_{(5)}$$ be the scheme obtained from $X$ by base change. 
The fiber over the generic point $(0)$ is $X_{\mathbb Q}$ and we are interested in extending $X_{\mathbb Q}$ to an abelian scheme over $\operatorname{Spec}\mathbb Z_{(5)}$.  

Let $X_{\mathbb F_5}$  be the fiber of $X_{(5)}$ over the closed point.  Then    $X_{\mathbb F_5}$ is given by the equation 
$$
y^2-x^3=0,
$$
which is singular at $(0:0:1)$.       We would like to describe a finite base change and a modification of the family that is smooth.

Consider the finite, degree $2$ morphism $$B':=\operatorname{Spec}\mathbb Z_{(5)}[\zeta]/(\zeta^2-5)\to \operatorname{Spec}\mathbb Z_{(5)}$$ induced by the extension $\mathbb Q(\sqrt[2]{5})/\mathbb Q$.  Pulling back $X_{(5)}$ we obtain a family $X'\to B'$ defined by the equation
$$
y^2-x^3-\alpha\zeta^4x-\beta\zeta^6=0.
$$
Making the change of coordinates $x\mapsto \zeta^2x$, $y\mapsto \zeta^3y$, we arrive at a family $\tilde X\to B'$ defined by 
$$
y^2-x^3-\alpha x-\beta =0.
$$
This is smooth over the closed point $(\zeta)\in B'$, and in fact there is a group law over $B'$.   Thus, after a finite, degree two, base change, we have modified our family to give an abelian scheme over the base.

\begin{rem}
It is interesting to note the connection with the geometric case.  The example above was constructed to be the analogue of the linear family over the Weyl cover (\S \ref{secmonobexa}): 
$$
y^2+x^3-(b_2^2+b_2+1) b_1^2 x-b_2(1+b_2) b_1^3 =0,
$$
where, roughly speaking,  we replaced $x_2$ with $y$, $x_1$ with $-x$, set $b_1=5$ and took $b_2$ general. We had seen that a degree two base change for the analytic family would allow for stable reduction, and this is exactly what we have found here in the arithmetic setting.  
\end{rem}

\begin{rem}
For completeness, we mention that the discriminant $\Delta$ and $j$-invariant of $X$ are
 $$\Delta=-16(4\alpha^3+27 \beta^2)(5^6)\ne 0 \ \ \text{ and } \ \ j=-1728(4\alpha)^35^6/\Delta.$$
Note that $j$ has non-negative valuation at $5$.  It is well known that from this data one can deduce that the family does not have abelian reduction at $5$, but does have potentially abelian reduction there   (see e.g.~\cite[VII 
Prop.~5.1, 5.5]{silverman}).
\end{rem}

\subsubsection{Monodromy}
Recall  that the analytic monodromy representation of the analogous family was given by the negative of the identity (\S \ref{secmonobexa}).
Although we have not introduced the algebraic monodromy operator, we make the following observation.  Since there is abelian reduction after a degree two base change, one can immediately conclude that the algebraic monodromy operator is a square root of the identity.  
One can then show the  action on non-trivial,  torsion points of sufficiently high order (relatively prime to $2$ and $5$) is non-trivial.  Thus  the monodromy operator is the negative of the identity, similar to the analogous analytic case.   For more on degenerations of elliptic curves and monodromy, we direct the reader to the discussion of the Kodaira--N\'eron classification in \cite[App.~C.15]{silverman}.

\subsection{Group scheme terminology}  We now review some of the basic terminology of group schemes, directing the reader to \cite[\S 4.1]{blr} and \cite[Ch.6]{git} for more details.  
For a scheme $B$, a \textbf{$B$-group scheme} is a group object in the category of $B$-schemes ($\operatorname{Sch}/B$).    
A standard example, which  we will use frequently, is  $\mathbb G_m=\operatorname{Spec}\mathbb Z[t,t^{-1}]$, with group law induced by the map
$$
\mathbb Z[t,t^{-1}]\to \mathbb Z[t,t^{-1}]\otimes_{\mathbb Z}\mathbb Z[t,t^{-1}] \ \ \text{ given by } \ \ t\mapsto t\otimes t.
$$
For an arbitrary scheme $B$, we define $\mathbb G_{m,B}$ by base extension, and the induced group law makes $\mathbb G_{m,B}$ a group object in the category of $B$-schemes. 
An \textbf{(affine) split $B$-torus $T$} is a $B$-group scheme that is isomorphic as a $B$-group scheme to a finite fibered product $\mathbb G_{m,B}\times _B\ldots \times_B\mathbb G_{m,B}$.  An \textbf{(affine) $B$-torus $T$} is a $B$-group scheme that is \'etale locally on $B$ a split torus.
We define \textbf{$B$-subgroup schemes} in the obvious way (see e.g.~\cite[p.98]{blr}).

An \textbf{abelian scheme over $B$} is a $B$-group scheme that is smooth and proper over $B$ with connected fibers.  It follows from the Rigidity Lemma that the group law of an abelian scheme is commutative (see e.g.~\cite[Pro.~6.1, Cor.~6.4, p.115-6]{git}).
It  is a well known result   that an abelian scheme over a field is projective  (\cite{weil}).

\begin{rem}
 In order to use the term variety consistently (within this paper), we  reserve the term \textbf{abelian variety} for an abelian group scheme over an algebraicaly closed field. (This is not standard, in that one usually does not require the field to be algebraically closed.)
\end{rem}

The best understood abelian schemes are Jacobians of curves.  Recall that associated to a smooth curve $X$ over an algebraically closed field, there is an abelian variety $JX$, called the Jacobian of $X$,  parameterizing degree zero line bundles on $X$.   We note in addition that associated to a family of smooth curves $X\to B$, there is an associated abelian scheme $JX_B$ over $B$, called the (relative) Jacobian of $X_B$, with geometric fibers that are the Jacobians of the associated curves.

  A \textbf{semi-abelian scheme} $G_B$ over $B$ is a smooth, separated, commutative $B$-group scheme such that each fiber $G_{B,b}$ over $b\in B$ is an extension of an abelian scheme $A_{b}$ by an affine torus $T_{b}$:
$$
0\to T_{b}\to G_{B,b}\to A_{b}\to 0.
$$
  We direct the reader to \cite[Cor.~2.11]{fc} for a statement on the global structure of a semi-abelian scheme.   Extensions of an abelian variety $A/k$ by a torus $T/k$ are classified (up to isomorphism as extensions) by  $\operatorname{Hom}\left(X(T),\widehat A\right)$, where $X(T)=\operatorname{Hom}(T,\mathbb G_{m,k})$ is the character group and $\hat A=\operatorname{Pic}^0(A)$ is the group of line bundles on $A$ algebraically equivalent to zero   (see e.g.~\cite[Thm.~6, p.184]{serregroupes}).

We now come to the topic of reduction.   Let $R$ be a DVR,  let $K$ be its field of fractions, and let $S=\operatorname{Spec} R$.  Let  $A_K$  be an abelian scheme over $\operatorname{Spec}K$.  
 We say that $A_K$ has \textbf{abelian (or good) reduction} (resp.~\textbf{semi-abelian reduction}) 
 if $A_K$ can be extended to a smooth, separated $S$-group scheme $G_S$ of finite type over $S$ such that the
 the fiber over the closed point $s\in S$ 
 is an abelian (resp.~semi-abelian) scheme over $s$.  We will say that  $A_K$ has \textbf{potentially abelian (or good) reduction} (resp.~\textbf{potentially semi-abelian reduction})  if there is a finite extension $K'$ of $K$,   
 so that the abelian scheme $A_{K'}$ obtained by base change has abelian (resp.~semi-abelian) reduction.

\subsection{N\'eron models}  N\'eron models provide a natural context for discussing the stable reduction theorem for abelian varieties.
  While the theory can be developed in more generality over Dedekind domains, we focus on the case of DVRs for simplicity.  
  
     As above, let $S=\operatorname{Spec} R$ be the spectrum of a DVR with fraction field $K$. 
   Let $X_K$ be a smooth, separated $K$-scheme of finite type.  
    A \textbf{N\'eron model} of $X_K$ is an extension $X_S$ of $X_K$ over $S$ that is a smooth, separated scheme of finite type, satisfying the following universal property: for any smooth $S$-scheme $Y_S$ and any $K$-morphism  $f_K:Y_K\to X_K$ there is a unique $S$-morphism $f_S:Y_S\to X_S$ extending $f_K$.  
    
$$
\xymatrix
{
&X_K\ar@{->}[rr] \ar@{-}[d]&&X_S \ar@{->}[dd]\\
Y_K\ar@{->}[rr] \ar@{->}[rd] \ar@{->}[ru]^{f_K}&\ar@{->}[d]&Y_S\ar@{->}[rd] \ar@{-->}[ru]^{f_S}_{\exists !}&\\
&\operatorname{Spec}K\ar@{->}[rr]&&S.\\
}
$$
If a N\'eron model exists, it is unique up to unique isomorphism.  
While N\'eron models do exist in a more general setting, we will focus here on the case of  abelian schemes.  The main theorem in this situation is:

\begin{teo}[N\'eron \cite{neron64}]
Let $A_K$ be an abelian scheme over the field of fractions $K$  of a DVR $R$.  Then $A_K$ admits a N\'eron model $X_S$ over $S=\operatorname{Spec}R$.    
\end{teo}

We direct the reader also to \cite[Cor.~2, p.16, Pro.~6, p.14]{blr} and Artin \cite{artin86}.  

\begin{rem}
From the universal property of the N\'eron model, it follows that the $K$-group scheme structure on $A_K$ extends uniquely to a commutative $S$-group scheme structure on $X_S$.      For group schemes, the condition that the N\'eron model be of finite type and separated is superfluous (e.g.~\cite[p.12, Rem.~7, p.14]{blr}).   Finally, it is a result of Raynaud that the N\'eron model of an abelian scheme is quasi-projective \cite[Thm.~VIII.2, p.120]{raynaudf}.
\end{rem}

\begin{rem}
The special fiber of a N\'eron model of an abelian scheme need not be connected.   One such example is given by a smooth plane cubic degenerating to a nodal plane cubic that is the union of a   line and a smooth conic.
	The special fiber $X_{S,s}$ of the N\'eron model can be computed using a result of Raynaud discussed in the remark below.  One can show $X_{S,s}$ fits into an exact sequence $$0 \to  \mathbb G_m \to  X_{S,s} \to \mathbb Z/2\mathbb Z \to 0$$ 
(see e.g.~Kass \cite[\S 4.3]{kass}).    
There is always, however, an open $S$-subgroup scheme  of a N\'eron model of an abelian scheme $A_K$ that extends $A_K$ and has connected central fiber.     We will denote this by $X_S^\circ$.
\end{rem}

\begin{rem} \label{remray}
As with abelian varieties, the best understood N\'eron models are those associated to  curves.
One of the main tools  is a theorem of Raynaud's, relating the N\'eron model of a Jacobian to the Picard functor.    The following is a weaker  version of the  theorem, given in Deligne--Mumford \cite[Thm.~2.5]{dm}, which is used in the proof of the stable reduction theorem for curves.
\emph{Assume the residue field of $R$ is algebraically closed.  In the notation above, let $C_S$ be a generically smooth family of nodal curves over $S$, with non-singular total space, and let $A_K$ be the Jacobian of the generic fiber.
Then the open $S$-subgroup scheme $X_S^\circ$ of the N\'eron model $X_S$ of $A_K$, described above, represents the  relative (connected component of the) Picard functor.}
\end{rem}

\begin{exa}
Assume the residue field of $R$ is algebraically closed.  
Consider a family $C\to S$ of smooth curves   degenerating to an irreducible, stable curve $C_s$ with a single node.  Let $C_s^\nu$ be the normalization of $C_s$, and assume that $C_s$ is obtained from $C_s^\nu$ by attaching points $p,q\in C_s^\nu$.    The family of curves $C_K$ over $K$ determines a principally polarized abelian scheme $X_K=JC_K$, the Jacobian of the curve.     The special  fiber of  the open $S$-subgroup scheme $X_S^\circ$ of the N\'eron model of $X_K$ is an extension
$$
0\to \mathbb G_m\to X_{S,s}^\circ\to JC_s^\nu \to 0
$$ 
determined by the data of the line bundle $\mathscr O_{C_s^\nu}(p-q)$. 
\end{exa}

\subsubsection{The group structure of the central fiber of the N\'eron model} 
In the notation above, we have seen that the N\'eron model of an abelian scheme $A_K$ is a commutative group scheme over $S$.    To get a handle on how   N\'eron models are connected to the question of semi-abelian reduction, we  will investigate the group structure on the central fiber of the N\'eron model using  a few basic facts from the theory of algebraic groups (see also Serre \cite{serregroupes}).

To begin, we recall Chevalley's theorem \cite{chevalley60,rosenlicht56} (see esp.~\cite{blr} and \cite[Thm.~1.1]{conrad02}): 
\emph{
Let $K$ be a field and let $G$ be a smooth, connected algebraic $K$-group.  Then there exists a smallest (not necessarily smooth) connected linear subgroup $L$ of $G$ such that the quotient $G/L$ is an abelian scheme over $K$.  Moreover, if $K$ is perfect, $L$ is smooth and its formation is compatible with change of base field.}

In other words, if $X_S$ is the N\'eron model of $A_K$, then the connected component of the identity in the central fiber $X_{S,s}^\circ$ fits into an exact sequence over $s$
$$
0\to L \to X_{S,s}^\circ\to A\to 0,
$$
where $L$ is a connected, commutative, linear $s$-group scheme  and $A$ is an abelian $s$-group  scheme.

We now turn our attention to the structure of  linear algebraic groups.    Let us pull-back to the algebraic closure $\bar k=\overline{\kappa(s)}$, and denote the resulting group schemes by $\bar L$, $\overline X_{S,s}^\circ$ and $\bar A$ respectively.  $\bar L$ being commutative, it is solvable (see e.g.~\cite[Def., p.59]{borel}).  There is the following standard theorem   (e.g.~\cite[Thm.~10.6, p.137]{borel}):  \emph{
For a connected, solvable, linear algebraic group $\bar L$ over an algebraically closed field $\bar k$, the subset of unipotent elements $\bar L_u$ is a (closed) connected, normal $\bar k$-subgroup, and the quotient is an affine torus.}
In the situation of the N\'eron model,   this torus can be obtained by pull-back from a torus over $\kappa(s)$. 
Thus  \emph{the (subgroup $X_{S,s}^\circ$ of the) N\'eron model is a semi-abelian scheme, if and only if  $\bar L_u$ is trivial}.

\begin{rem} The standard way to assert that $\bar L_u$ is trivial is to assert that the  unipotent radical of $\bar L$ is trivial.  Indeed, for a connected, solvable group $G$ over an algebraically closed field, the radical $\mathscr RG$ is equal to the group  $G$ (the radical  is the largest connected, solvable, normal subgroup; e.g.~\cite[p.157]{borel}).  Thus in this case the unipotent radical $(\mathscr RG)_u=:\mathscr R_uG$  (the set of unipotent elements of the radical) is equal to the set $G_u$.   \end{rem}

\subsection{The stable reduction theorem}  

The stable reduction theorem plays a central role in the study of abelian varieties, and also in the study of algebraic curves.   In light of the results of N\'eron, and the basic structure theorems for algebraic groups, the stable reduction theorem 
states that after a generically finite base change, the unipotent radical of the  central fiber of the N\'eron model can be made  trivial.

As described in the introduction to  \cite{dm}, the stable reduction theorem was first proved independently by Grothendieck and Mumford in characteristic zero.  Grothendieck's proof used the theory of \'etale cohomology, while Mumford's proof was derived from a stable reduction theorem for curves (in characteristic zero).   Grothendieck then  extended his proof to all characteristics in \cite[Thm.~6.1, p.21]{sga7I} and Mumford provided an independent proof in characteristics other than  $2$ using the theory of theta functions.

\begin{teo}[Grothendieck--Mumford Stable Reduction Theorem]  Let $S=\operatorname{Spec}R$ be the spectrum of a DVR with fraction field $K$.  
An abelian variety $A_K$ over $K$ has potential semi-abelian reduction over $R$.
\end{teo}

In fact the theorem can be  stated more generally for the case where  $A_K$ is a semi-abelian scheme.  
We refer the reader also to \cite[Thm.~1, p.180]{blr},  and to  \cite[Thm.~2.6, p.9]{fc}.  Grothendieck's proof relies on the  following frequently cited result, which   was the basis of the monodromy obstruction computation in \S \ref{secexa}.

\begin{pro}[Grothendieck {\cite[Prop.~3.5, p.350]{sga7I}}]
In the notation above, $A_K$ has abelian (resp.~semi-abelian) reduction if and only if the monodromy representation  is trivial (resp.~unipotent).
\end{pro}

We direct the reader also to \cite[Thm.~5, p.183]{blr}.   
The Grothendieck--Mumford Stable Reduction Theorem follows from the proposition and the Monodromy Theorem  after the observation  (see Remark \ref{remqutou}) that  quasi-unipotent  monodromy can be made unipotent after a finite base change.

\subsection{Alexeev's space of stable semiabelic pairs}   
One would like to derive from the Grothendieck--Mumford Stable Reduction Theorem a properness statement for a moduli space.  This serves as motivation to introduce Alexeev's compactification \cite{alexeev} of the moduli space of principally polarized abelian varieties, where such a statement holds.

Let us recall some definitions from \cite{alexeev} (we also direct the reader to Olsson  \cite{olsson} for a related moduli problem).  First a reduced scheme $X$ is said to be \textbf{semi-normal} if given any  proper, bijective morphism $f:X'\to X$  from a reduced scheme $X'$ satisfying  the property that  $\kappa(f(x'))\to \kappa(x')$ is an isomorphism for all $x'\in X$,  then $f$ is an isomorphism (see e.g.~\cite[\S7.2]{kollar}, \cite[1.1.6]{alexeev}).  
For instance, a nodal curve is semi-normal, while a cuspidal curve is not.

A \textbf{stable semiabelic variety} (\cite[1.1.5]{alexeev}) is a semi-normal, equidimensional, reduced scheme  $X$ over an algebraically closed field $k$, together with an action of a connected semi-abelian scheme $G/k$ of the same dimension, such that there are only finitely many orbits for the $G$-action, and the stabilizer group scheme of every point of $X$ is connected, reduced and lies in the toric part of $G$.

A \textbf{polarized stable semiabelic variety} (\cite[1.1.8]{alexeev}) is a projective stable semiabelic variety together with an ample invertible sheaf $L$.  The degree of the polarization is defined as $h^0(L)$.   A \textbf{stable semiabelic pair} $(X,\Theta)$ consists of a polarized stable semiabelic variety $X$ with ample invertible sheaf $L$ together with a section $\theta\in H^0(X,L)$ that does not vanish on any $G$-orbits.   So in total, for a stable semiabelic pair, we have the data $(X,G,L,\theta)$.   We take $\Theta$ to be the zero set of $\theta$, and use the shorter notation $(X,\Theta)$ to indicate the connection to polarized abelian varieties.

We now make the relative definition.  For a scheme $B$, a \textbf{stable semiabelic pair over $B$}, denoted  $(X_B,\Theta_B)$, is the data $$(X_B,G_B,L_B,\theta_B)$$ where $X_B\stackrel{\pi_B}{\to} B$ is a projective, flat morphism, $G_B$ is a semi-abelian scheme over $B$ acting on $X_B$, $L_B$ is a relatively ample line bundle on  $X_B$, $\theta_B\in H^0(B,\pi_*L_B)$, and the restriction of this data to every geometric point $\bar b\to B$  is a stable semiabelic pair over $\bar b$. (\cite[p.617]{alexeev}).  One can show that $\pi_*L_B$ is locally free and that this push forward commutes with arbitrary base change.  The degree is defined to be the rank of $\pi_*L_B$.

 For brevity, we do not give a precise definition of Alexeev's stack $\bar{\mathcal A}_{g}^{A}$, and say only that it is a substack of the stack of all stable semiabelic pairs of degree $1$ and dimension $g$.   The stack $\bar{\mathcal A}_{g}^{A}$ contains a component that has ${\mathcal A}_{g}$, the moduli stack of principally polarized abelian varieties of dimension $g$,  as a dense open substack.
 Alexeev proves \cite[Thm~5.10.1]{alexeev} that   $\bar{\mathcal A}_{g}^{A}$ is a proper, algebraic (Artin) stack over $\mathbb Z$ with finite diagonal.   Moreover, the stack admits a coarse moduli space, with a component that has normalization isomorphic to the second voronoi compactification $\bar A_g^{Vor}$ \cite[Thm.~5.11.6, p. 701]{alexeev}.  To establish properness, 
 Alexeev proves the following stable reduction theorem for semiabelic pairs.

\begin{teo}[Alexeev {\cite[Thm.~5.7.1, p.692]{alexeev}}] Let $S=\operatorname{Spec} R$ be the spectrum of a DVR with fraction field $K$.  Let $(X_K,\Theta_K)$ be a stable semiabelic pair over $K$.   Then there is a finite extension $K'$ of $K$, so that taking $R'$ to be the integral closure of $R$ in $K'$  and setting $S'=\operatorname{Spec}R'$, 
there exists a stable semiabelic pair $(X_{S'},\Theta_{S'})$ over $S'$ extending the pull-back $(X_{K'},\Theta_{K'})$.
Morover, the extension $(X_{S'},\Theta_{S'})$ is unique up to isomorphism.  
\end{teo}

\begin{rem}
As a consequence, the central fiber $(X_{s'},\Theta_{s'})$ of  $(X_{S'},\Theta_{S'})$ is determined up to isomorphism by $(X_K,\Theta_K)$. 
\end{rem}

\begin{exa}
Assume the residue field of $R$ is algebraically closed.  
Consider a family $C\to S$ of smooth curves   degenerating to an irreducible, stable curve $C_s$ with a single node.  Let $C_s^\nu$ be the normalization of $C_s$, and assume that $C_s$ is obtained from $C_s^\nu$ by attaching points $p,q\in C_s^\nu$.    The family of curves $C_K$ over $K$ determines a principally polarized abelian scheme $(X_K,\Theta_K)$, where  $X_K=JC_K$ is the Jacobian of the curve.   Let $(X_{S'},\Theta_{S'})$ be a stable reduction of $X_K$. 

 The central fiber can be described as follows.   The degree zero Picard functor applied to $C_S$ determines a semi-abelian scheme over $S$.  The central fiber is the semi-abelian scheme 
$$
0\to \mathbb G_m\to G\to JC_s^\nu \to 0
$$ 
determined by the data of the line bundle $\mathscr O_{C_s^\nu}(p-q)$.  The group $G$ can be completed to a $\mathbb P^1$-bundle over $JC_s^\nu$, with sections $\sigma_0$ and $\sigma_\infty$.  The fiber $X_{s'}$ is obtained from this projective bundle by gluing the sections transversally, after shifting by $\mathscr O_{C_s^\nu}(p-q)$. Note that  $G$ acts on this space.    We direct the reader to \cite{alexeevtor} for more details, as well as a description of the polarization (see also \cite{mumford83}).
\end{exa}

\section{Curves} \label{seccurves}

In this section we consider stable reduction for curves.  
We begin with the Deligne--Mumford Stable Reduction Theorem.  
The main focus is on reviewing the connection between stable reduction for curves and stable reduction for abelian varieties.  We also review a well known proof in characteristic zero using the semi-stable reduction theorem, to motivate work of Kollar--Shepherd-Barron--Alexeev, discussed later.  Finally, we consider some recent  ``alternate'' stable reduction theorems for curves, which have connections to the Hassett--Keel program.

\subsection{Deligne--Mumford stable reduction}
Recall that a stable curve $X$ over an algebraically closed field is a pure dimension $1$, reduced, connected, complete scheme of finite type, with at worst nodes as singularities, and with finite automorphism group.   The genus is defined as $g=h^1(X,\mathscr O_X)$.   For a scheme $B$, a stable curve $X/B$ is a proper, flat morphism $X\to B$ whose geometric fibers are stable curves.

\begin{teo}[Deligne--Mumford Stable Reduction \cite{dm}] Let $S=\operatorname{Spec} R$ be the spectrum of a DVR with fraction field $K$.  Let $X_K$ be a stable curve over $K$.  
 Then there is a finite extension $K'$ of $K$, so that taking $R'$ to be the integral closure of $R$ in $K'$  and setting $S'=\operatorname{Spec}R'$, 
there exists 
 a stable curve $X_{S'}$ over $S'$ extending the pull-back $X_{K'}$.
Morover, the extension $X_{S'}$ is unique up to isomorphism.   
\end{teo}

\begin{rem}
As a consequence, the central fiber $X_{s'}$ of $X_{S'}$ is determined up to isomorphism by $X_K$.
\end{rem}

In characteristic $0$, the theorem is due to Mayer--Mumford \cite{woodshole}.  It appears that semi-stable reduction for curves in characteristic $0$ was well known for some time (e.g.~\cite[p.VIII]{mumetal}).  
The class of stable curves was then defined in \cite[Def., p.7]{woodshole} (see also \cite[p.228]{git}), and the properness of the associated moduli space was asserted there.  The properness is equivalent to the stable reduction theorem, which one establishes from the existence of a semi-stable reduction,  and the birational geometry of surfaces (we outline a well known argument below).

 In positive characteristic, the first proof was due to Deligne--Mumford \cite{dm},  and was made via the stable reduction theorem for abelian varieties.  
The outline of this proof is as follows.  Take $X_K$ smooth for simplicity and consider the  Jacobian $J_K/K$.  The Grothendieck--Mumford Stable Reduction theorem implies this extends to a family of semi-abelian varieties, at least after a finite base change.  To complete the proof, it is then  shown:

\begin{teo}[Deligne--Mumford {\cite[Thm.~2.4]{dm}}] \label{teodmII} 
A family of stable curves $X_K/K$ extends to a family of stable curves over $S$ if and only if the associated family of Jacobians $J_K/K$ has semi-abelian reduction over $S$.
\end{teo}

A key point of the proof is  the result of Raynaud's mentioned in Remark \ref{remray}.  
The fact that stable reduction for the family of curves implies stable reduction for the family of Jacobians essentially follows directly from Raynaud's result.  Using this half of Theorem \ref{teodmII}, Mumford observed \cite[p.75]{dm}  that  the stable reduction theorem for  abelian varieties  can be deduced from  the stable reduction theorem for curves.

We direct the reader to \cite[p.182]{blr} for a more detailed discussion.   
The outline of the argument is as follows.  An abelian scheme $A_K$ can be viewed as the quotient of a product of Jacobians; i.e. 
$$
0\to A_K'\to J_K\to A_K\to 0
$$
where $A_K'$ is an abelian scheme, and $J_K$ is a product of Jacobians (e.g.~Serre \cite[Cor.~p.180]{serregroupes}).  One then shows that in general, for such an extension of abelian schemes, $J_K$ has semi-abelian reduction if and only if $A_K$ and $A_K'$ do, completing the proof.    
Finally, we note that 
Artin--Winters have given another  proof of stable reduction  for curves in positive characteristic, which does not rely on the stable reduction theorem for abelian varieties \cite{artinwinters}.

\begin{rem}\label{remtrivmon}  \ 
While unipotent monodromy for a one-parameter family of smooth curves implies the family extends to a family of stable curves, trivial monodromy does not necessarily imply that the central fiber of the extension is smooth.  For instance, a one-parameter family of smooth curves degenerating to a singular, stable curve of compact type will have associated   Jacobian that is an abelian scheme over the base.  Grothendieck's theorem implies that the monodromy of the family will be trivial.  We direct the reader to Oda  \cite[Thm.~10]{oda93} for a statement concerning a related monodromy invariant that detects when a one-parameter family can be extended to a smooth curve.  
\end{rem}

\begin{rem}
There is a stable reduction theorem for pointed curves as well.   
Pointed curves provide a natural introduction to  the important topic of  moduli spaces of pairs.  For the sake of brevity, we have generally avoided this topic in the presentation here.  It will, however, be of central importance in \S \ref{seckol} on slc models, and it is worth noting this example as a precursor.    
\end{rem}


\subsubsection{Stable reduction in characteristic $0$} 
To motivate some of the other examples considered in this survey (following the approach of Koll\'ar--Shepherd-Barron--Alexeev), it is instructive to sketch a  proof of a special case of the stable reduction theorem for curves in characteristic $0$, using the semi-stable reduction theorem.   The goal is to emphasize the role of semi-stable reduction  and relative canonical models.

\begin{proof}[Sketch of stable reduction for curves in characteristic $0$]  \ \ \  \ \  For simplicity we consider the case of a  smooth family of curves $$\pi_K:X_K\to \operatorname{Spec}K$$  (of genus $g$) over the generic point of $S=\operatorname{Spec}R$, the spectrum of a DVR.    Complete this to a family of schemes $\pi:X\to S$.  Applying the semi-stable reduction theorem, one obtains  after a finite base change a family of nodal curves $\pi':X'\to S'$.  
The relative canonical model $$\operatorname{Proj}_{S'} \bigoplus_n \pi'_*\left(\omega_{X'/S'}^{\otimes n}\right)\to S'$$   is a family of stable curves extending the pull back of $X_K$.  Let us denote this by $\pi^c:X^c \to S'$.    
Note that the relative canonical model of  $X^c /S'$ is again $X^c /S'$.   (In short, we established the valuative criterion \eqref{eqnvcp}.)

Let us now show that the extension is unique up to isomorphism (i.e.~that the valuative criterion of separateness  \eqref{eqnvcs} holds).
 To do this,  
suppose there were two stable reductions $\pi^c_1:X^c _1\to S'$ and $\pi^c_2:X^c _2\to S'$.   
The surfaces $X^c _1$ and $X^c _2$ are birational by construction.  The claim is they are in fact isomorphic over $S'$.   
 We outline the following standard proof of this in  order to motivate similar statements in other settings.

Resolving the singularities of the surfaces, resolving the resulting birational map of smooth surfaces, and applying the Semi-stable Reduction Theorem again if necessary, we may assume there is a diagram:
$$
\xymatrix
{
&Z \ar@{->}[ld]_{\phi_1} \ar@{->}[rd]^{\phi_2}&\\
X^c _1 \ar@{->}[rd]_{\pi^c_1}&&X^c _2 \ar@{->}[ld]^{\pi^c_2}\\
&S'&\\
}
$$  
where $\phi_1,\phi_2$ are  sequences of blow-ups,  $Z$ is a smooth surface, and  $Z/S'$ is a family of nodal curves.  One can show that $(\pi^c_{1})_*\omega_{X^c _1/S'}^{\otimes n} \cong (\pi^c_1\circ \phi_1)_*\omega_{Z/S'}^{\otimes n}$  (see e.g.~\cite[Ex.~3.108, p.156, p.84]{hm}) and similarly for the other side of the diagram. 
It follows that 
$$
(\pi^c_{1})_*\omega_{X^c _1/S'}^{\otimes n} \cong (\pi^c_1\circ \phi_1)_*\omega_{Z/S'}^{\otimes n}\cong (\pi^c_2\circ \phi_2)_*\omega_{Z/S'}^{\otimes n}\cong (\pi^c_{2})_*\omega_{X^c _2/S'}^{\otimes n}.
$$
Thus the relative canonical models of $X^c _1/S'$ and $X^c _2/S'$ agree, so in fact $X^c _1$ and $X^c _2$ are isomorphic over $S'$. 
\end{proof}

\subsection{Other stable reduction theorems for curves} \label{secaltsrc}

Recently, especially in connection with  the Hassett--Keel program (see \S \ref{secade}), there has been interest in understanding alternate compactifications of (open subsets of) $M_g$.  From the perspective of stacks, this is the question of determining alternate proper stacks $\overline {\mathcal M}_g^{Alt}$ that contain (open substacks of) $\mathcal M_g$ as an open substack.    From the perspective of stable reduction, this is the problem of defining classes of curves for which a stable reduction theorem holds.  We direct the reader to Smyth \cite{smyth09} for more details (see also \cite{asv10},\cite{as12} where a notion weaker than properness is considered). Here we consider Schubert's space of pseudo-stable curves.

We recall the definitions from \cite{schubert} (see also \cite{hh}).  
A pseudo-stable curve $X$ over an algebraically closed field $k$ is a pure dimension $1$, reduced, connected, complete scheme of finite type, with at worst nodes and cusps as singularities, such that the canonical line bundle is ample and every sub-curve of genus $1$ meets the rest of the curve in at least two points  (\cite[Def.~p.297]{schubert}, \cite[p.4473]{hh}).    
We note that for $g\ge 3$, pseudo-stable curves have finite automorphism groups (\cite[p.312]{schubert}).  
The genus is defined as $g=h^1(X,\mathscr O_X)$.  
 For a scheme $B$, a pseudo-stable curve $X/B$ is a proper, flat morphism $X\to B$ whose geometric fibers are pseudo-stable curves. 
 
There is the following stable reduction theorem for this class of curves.

\begin{teo}[Pseudo-Stable Reduction \cite{schubert}] Let $S=\operatorname{Spec} R$ be the spectrum of a DVR with fraction field $K$.  Let $X_K$ be a genus $g\ge 3$, pseudo-stable curve over $K$.  
 Then there is a finite extension $K'$ of $K$, so that taking $R'$ to be the integral closure of $R$ in $K'$  and setting $S'=\operatorname{Spec}R'$, 
there exists 
 a pseudo-stable curve $X_{S'}$ over $S'$ extending the pull-back $X_{K'}$.
Morover, the extension $X_{S'}$ is unique up to isomorphism.   
\end{teo}

For details we refer the reader to \cite[\S 2]{hh} where the results in \cite{schubert} are translated into the language of stacks.
To be precise,  for each $g\ge 3$, define a stack $\overline {\mathcal M}_g^{ps}$ whose objects are families of genus $g$, pseudo-stable curves, and whose morphisms are given by pull-back diagrams.  The results in  \cite{schubert}  establish that 
 $\overline{\mathcal M}_g^{ps}$ is a proper DM stack.


\section{Stable reduction in higher dimensions}
\label{secsrhd}

We now consider the problem of stable reduction for higher dimensional varieties.  
In general, determining a proper moduli stack, or even a separated moduli stack, is quite difficult.
The literature on this topic is vast, and we direct the reader to Viehweg \cite{viehweg95}, Alexeev \cite{alexeev96, alexeev96b}, Koll\'ar--Shepherd-Baron \cite{KSB} and Koll\'ar \cite{kollar90,kollar11,kollarbook}.

The  focus of this section is to outline the approach of Koll\'ar--Shepherd-Barron--Alexeev (KSBA). The basic idea is to utilize relative log-canonical models of the semi-stable reduction to obtain the  ``stable'' reduction.  
We discuss recent results of Koll\'ar \cite{kollar11} in \S \ref{seckol} that extend the results cited above to give a stable reduction theorem for  canonically polarized varieties of any dimension.  In \S\ref{secnkd} we discuss an example indicating a few of the difficulties that  arise for varieties with negative kodaira dimension, and in  \S\ref{seck3} we  discuss the case of  $K3$ surfaces, for which a stable reduction theorem is not known.

\subsection{Negative Kodaira dimension} \label{secnkd}

It has long been understood that moduli spaces of non-canonically polarized schemes  are in general,  poorly behaved.  In particular, in this subsection, we consider the case of varieties with negative Kodaira dimension.

One of the immediate problems that arises is the existence of varieties with non-discrete, affine automorphism groups.  This immediately precludes the separateness of any moduli stack containing such varieties (since the diagonal of the stack could not be proper).   We direct the reader to Kov\'acs \cite[\S 5.D]{kovacsypg09} for a number of examples and for further discussion.

 For the convenience of the reader, we recall the following elementary example (\cite[Exa.~5.10]{kovacsypg09}).
Let us show in concrete terms, that any moduli stack containing $\mathbb P^1_k$  will fail  the valuative criterion for separatedness.  
Let $X=Y=\operatorname{Proj}_{k[t]}k[X_0,X_1,t]=
\left(\mathbb P^1_k\times \mathbb A^1_k\right)/\mathbb A^1_k$.  Over the open set $U=\operatorname{Spec}k[t]_t$, there is an isomorphism $X_U\to Y_U$ given by $([a_0:a_1],b)\mapsto ([ba_0:a_1],b)$.  This clearly does not extend to an isomorphism over $\mathbb A^1_k$.  Passing to the DVR $R=k[t]_{(t)}$, 
one sees the valuative criterion for separateness fails.    In fact, one can show that the valuative criterion for separateness fails in this example even if one considers polarizations (see \cite[Exa.~5.10]{kovacsypg09}).

For contrast, we direct the reader to Matsusaka--Mumford \cite[Thm.~2]{matmum64}  for a general result on separateness of moduli spaces of polarized manifolds that are not uni-ruled.  
 Finally, we point out that there do exist separated moduli spaces of certain uni-ruled varieties.  For instance, there are separated moduli spaces of Fano hypersurfaces of degree at least $3$ (\cite[Prop.~4.2, p.79]{git}), and we direct the reader to  Benoist \cite[Thm.~1.6]{benoistsep} for some recent results on separateness of moduli stacks of Fano  complete intersections.

\subsection{$K3$ surfaces} \label{seck3}

As another indication of the difficulties in establishing stable reduction theorems, we briefly  discuss the case of  $K3$ surfaces.    We work over $\mathbb C$.    Recall  a $K3$ surface is  a smooth, complex, projective surface $X$ with $K_X\cong \mathscr O_X$ and $q:=h^1(\mathscr O_X)=0$.   A polarized $K3$ surface is a  pair $(X,L)$ with $L$ an ample line bundle.  The degree of the polarization is defined to be $d:=L.L$.  Via  Hodge theory, one can construct a moduli space of polarized $K3$ surfaces of degree $d$, which we will denote by $ F_d^\circ$, together with a period map 
$$F_d^\circ\to \mathcal D_d/\Gamma_d,$$
 where $\mathcal D_d$ is a $19$-dimensional, symmetric homogeneous domain of type $IV$, and $\Gamma_d$ is an arithmetic group.  The morphism is injective (see \cite{loope80}) and has image equal to the complement of the quotient of an arithmetic hyperplane arrangement. 
   Note  that including $K3$ surfaces with $ADE$ singularities, one obtains a moduli space $F_d$ isomorphic to $\mathcal D_d/\Gamma_d$  \cite{kulikov77,perspink81}.
  We refer the reader to  \cite[\S 2.5, Thm.~2.9]{ghsK3ISM}, which includes  a concise overview of these results (and determines the Kodaira dimension of these spaces), and \cite[VIII]{barth} for more details and references. 
  
 The Satake--Bailly--Borel compactification $F_d^*:=(\mathcal D_d/\Gamma_d)^*$,  as well as the toroidal compactifications $\overline F_d:=\overline{\mathcal D_d/\Gamma_d}$, provide projective compactifications of the moduli of $K3$ surfaces.   It is not known whether any of these are the coarse moduli space for some proper stack of degenerations of $K3$ surfaces.

The first step towards constructing such a proper moduli space would be a stable reduction theorem.  A result  in this direction is a refined semi-stable reduction theorem due to  Kulikov and Persson--Pinkham.

\begin{teo}[Kulikov {\cite{kulikov77}}, Persson--Pinkham {\cite[Thm., p.45]{perspink81}}]
 A family of $K3$ surfaces over the punctured disk  $X^\circ\to S^\circ$ admits a semi-stable reduction $X\to S$ with central fiber $X_0$ (a reduced, snc scheme) \emph{such that   $K_{X_0}\equiv 0$}.    
\end{teo}

An algebraic version due to Shepherd--Barron \cite[Thm.~2,  p.136]{SBext81} for families of polarized $K3$ surfaces provides a projective completion of the family, where the central fiber has slc  (rather than snc) singularities.     Some results  due to Shah \cite{shah} using GIT constructions have provided some ``weak'' forms of stable reduction in some specific cases (in the sense of GIT; see \S\ref{secgit}).
Unfortunately, none of these provide a stable reduction theorem for $K3$ surfaces, even in the polarized case.
   We refer the reader to \cite{fm83}  for a more extensive discussion of  the topic.

\begin{rem}
Recently, combining the Shepherd--Barron result  \cite[Thm.~2,  p.136]{SBext81} with the KSBA strategy, Laza \cite[Thm.~2.11]{laza12} has constructed a proper moduli space of stable (slc) $K3$ pairs $(X,\Delta)$.   Essentially, rigidifying the moduli problem further by choosing sections of the polarizations with mild singularities, a stable reduction theorem is possible.  We refer the reader to Laza \cite{laza12} for more details. 
\end{rem}

\subsection{Canonically polarized varieties} \label{seckol}
We now consider stable reduction for canonically polarized varieties.    The main result is a recent theorem of Koll\'ar \cite{kollar11}, which states that stable reduction holds for slc models.  We begin by recalling some of the definitions.

\subsubsection{Preliminaries} In this section, 
    \emph{all schemes will be taken to be reduced, of finite type over $\mathbb C$, and all points will be taken to be closed points, unless otherwise stated}.      
 Recall a node of an equidimensional scheme $X$ of dimension $n$ is a point $x\in X$ such that $\widehat {\mathscr O}_{X,x}\cong \mathbb C[[x_1,\ldots,x_{n+1}]]/(x_1x_2)$ as $\mathbb C$-algebras.  
$X$ is said to have at worst nodes in codimension $1$ if there exists an open subset $V\subseteq X$ with $\operatorname{codim}_X(X-V)\ge 2$, with the property that for all $x\in V$, $x$ is either a non-singular  point, or a node.   
For a scheme $X$ that is $S_2$, $X$ is nodal in codimension one if and only if, in codimension one,  it  is both semi-normal and Gorenstein  (see Koll\'ar  \cite[\S 5.1, p.196]{kollarsing}).

We will want to discuss divisors on reducible, equidimensional, reduced schemes $X$.   A Weil divisor $D$  on such a scheme is a finite, formal, integral, linear combination of (not necessarily closed) points  $E\in X$ such that $\mathscr O_{X,E}$ is a DVR.   There is a notion of linear equivalence for such divisors obtained  via  Weil divisorial subsheaves; we direct the reader to Corti \cite[(16.1.1), (16.2.2), p.171-2]{corti}.    A $\mathbb Q$-divisor is defined similarly, with $\mathbb Q$-coefficients.    A $\mathbb Q$-divisor $D$ on $X$ is said to be $\mathbb Q$-Cartier if there exists an $m\in \mathbb N$ such that $mD$ is the Weyl divisor associated to a Cartier divisor.  

For   $X$ a projective scheme,  we denote by $\omega_X^\bullet$ the dualizing complex.  We set $\omega_{X}:=h^{-n}(\omega_X^\bullet)$, and call this the canonical sheaf of $X$.   
If $X$ is Gorenstein in codimension one, then associated to $\omega_{X}$ is a linear equivalence class of Weil divisors (see e.g.~\cite[(16.3.3), p.173]{corti}).  We denote this equivalence class by $K_X$ and call it the canonical divisor (class).   We direct the reader also to \cite[\S 5.5]{kollarmori} and  \cite[Def.~1.6, p.14]{kollarsing}  for more discussion.

\begin{rem}
In order to limit the length of this survey, we have suppressed the notion of a pair in most of the topics covered.     However, the utility of parameterizing varieties together with a distinguished divisor  goes back at least to the case of principally polarized abelian varieties, where the canonical bundle is trivial, and one substitutes the theta divisor in its place to provide a natural rigidity to the problem.  Recently it has become clear that in many other situations it can be beneficial to consider pairs $(X,\Delta)$ where $X$ is a variety, and $\Delta$ is an effective  divisor so that $K_X+\Delta$ is ample (see especially the work of Koll\'ar and Alexeev cited above).  The notion of pairs will be central in what follows.
\end{rem}

\subsubsection{Semi log canonical models} 
 
We start by recalling  the definition of log canonical pairs. 
Let \emph{$X$ be a projective, reduced, equidimensional, $S_2$ scheme and let $\Delta$ be an effective $\mathbb Q$-divisor on $X$ such that $K_X+\Delta$ is $\mathbb Q$-Cartier}.
With these assumptions, we say that the pair $(X,\Delta)$ is \textbf{log canonical (lc)} if $X$ is smooth in codimension one (or equivalently $X$ is normal) and  there exists a  log resolution $f:Y\to X$ of $(X,\Delta)$ such that 
\begin{equation}\label{eqnlc}
K_Y=f^*(K_X+\Delta)+\sum_i a_i E_i,
\end{equation}
where the $E_i$ are $f$-exceptional divisors and   $a_i\geq -1$ for every $i$.   Note that the equality in  \eqref{eqnlc} is   $\mathbb Q$-equivalence of $\mathbb Q$-Cartier divisors.
 We say that $X$ is \textbf{lc}  if the pair $(X,0)$ is lc, where $0$ is the trivial divisor (see e.g.~\cite{kollarmori} for more discussion).
 
With the assumptions above (in italics), we say that the pair $(X,\Delta)$ is \textbf{semi log canonical (slc)}   if
 $X$ is nodal in codimension one (or equivalently, in codimension one $X$ is seminormal and Gorenstein), $K_X+\Delta$ is $\mathbb Q$-Cartier, and if $\nu:X^{\nu}\to X$ is the normalization of $X$ and $\Theta$ is the $\mathbb Q$-Weil divisor on $X$ given
by
\begin{equation}\label{eqntheta}
K_{X^{\nu}}+\Theta=\nu^*(K_X+\Delta),
\end{equation}
then the pair $(X^{\nu}, \Theta)$ is lc.   Note that the equality in \eqref{eqntheta} is an equivalence, for which we refer the reader to  \cite[(5.7.5), Def.-Lem.~5.10]{kollarsing}.  
 We say that $X$ is \textbf{slc}  if the pair $(X,0)$ is slc, where $0$ is the trivial divisor (see e.g.~Abramovich--Fong--Koll\'ar--McKernen \cite{AFKM},  Fujino \cite{fujino00}, Koll\'ar \cite{kollar11} and Koll\'ar \cite[\S 5.2]{kollarsing} for more discussion).

A \textbf{semi log canonical model} (slc model) is an slc scheme $X$ such that  $K_X$ is ample (\cite[Def.~15]{kollar11}).  In particular,  if $X$ is smooth, then it is an slc model if and only if it is canonically polarized.  A  motivation for this definition also comes from the cases of curves and surfaces.  A semi log canonical model of dimension one is a stable curve of genus $g\ge 2$.  A result of Koll\'ar--Shepherd-Barron \cite[Cor.~5.7]{KSB}, Koll\'ar  \cite[Cor.~5.6]{kollar90} and Alexeev \cite{alexeev94} establishes that there is a proper moduli space of semi log canonical models of dimension two (with fixed invariants $K_X^2$ and $\chi(\mathscr O_X)$).  The valuative criterion for properness of the moduli space can be established with an appropriate stable reduction theorem.

\subsubsection{Koll\'ar's stable reduction theorem}

Recently Koll\'ar has established a stable reduction theorem for semi log canonical models of any dimension.   The full statement would require introducing the notion of relative semi log canonical models (and in particular the notion of reflexive hulls on non-normal schemes), which we omit  (see \cite[Def.~28, 29]{kollar11}).   Below we state a weaker version of this stable reduction theorem, where the generic fiber of the family is lc.  We sketch the parts of the proof that are  formally  similar to the proof of the stable reduction theorem that we sketched in the case of curves.

\begin{teo}[{Koll\'ar \cite[5.38]{kollar11}}]\label{teokollar} 
  Let $B$ be on open subset of a non-singular curve over $\mathbb C$, fix a point $o\in B$, and set $U=B-\{o\}$.  Suppose that 
$$
\pi:X\to B
$$
is a flat, projective, morphism with connected fibers, such that the restriction 
$
\pi_U:X_U=X|_U\to U
$
has lc fibers, and has $\pi_U$-ample relative dualizing sheaf $\omega_{X_U/U}$.    Then there is a finite base change $f:B'\to B$, with $B'$ non-singular and $f^{-1}(o)$ a single point $o'$, and a scheme 
$$\pi^c:X^c\to B'$$ such that $X^c$ and $B'\times_BX$ are isomorphic over $U':=B'-o'$, and the fiber $X^c_{o'}=(\pi^c)^{-1}(o')$ is an slc model.  Moreover, the extension $X^c/B'$ is unique up to isomorphism.  
\end{teo}

\begin{proof}[Sketch of the proof]  We sketch the proof in the case where the family is generically smooth.  For the general case  we refer the reader to \cite[\S 7]{haconxu} and \cite{kollarbook}.  
So, suppose that 
$
\pi:X\to B
$
is a flat, projective, morphism with connected fibers, such that the restriction 
$
\pi_U:X_U=X|_U\to U
$
is smooth, and has $\pi_U$-ample relative dualizing sheaf $\omega_{X_U/U}$.
From the Semi-stable Reduction Theorem we obtain a diagram as in \eqref{eqnsrdiag} including a morphism $\pi':X'\to B'$  satisfying the conclusions of Theorem \ref{teokollar}, except that the central fiber of the morphism $\pi':X'\to B'$, which is  normal crossing, may not have ample canonical class. 
(In the general case, where the fibers of $\pi_U$ are only assumed to be lc, rather than smooth, one applies a more general version of the Semi-Stable Reduction Theorem, e.g.~Theorem \ref{T:AbrKar}, below;  note, however,  that the generic fiber of the family obtained in this way may be different from the generic fiber of the original family.)

 From this point, motivated in part by the case of curves,  one considers 
$$
X^c:=\operatorname{Proj}_{B'}\left(\bigoplus_{k=0}^\infty \pi'_*\left(\omega_{X'/B'}^{\otimes k}\right)\right).
$$
A result of Birkar--Cascini--Hacon--McKernan \cite[Thm.~1.2 (3)]{bchm}   implies that the sheaf $\bigoplus_{k=0}^\infty \pi'_*(\omega_{X'/B'}^{\otimes k})$ is finitely generated as an $\mathscr O_{B'}$-algebra.  (See Hacon--Xu \cite{haconxu} for the general case, where the fibers of $\pi_U$ are only assumed to be lc.)
The projection $\pi^c:X^c\to B'$ is a projective morphism that  agrees with $B'\times_B X \to B'$ over $U'$.   One can show that 
the central fiber of $\pi^c$ is an slc model (\cite{kollar11}).

It remains to show that  $\pi^c:X^c\to B'$ is unique up to isomorphism.  
  One does this by establishing that  any other projective morphism $\hat \pi^c:\hat X^c\to B'$ with $K_{\hat X^c}$  a $\mathbb Q$-Cartier divisor, which  agrees with $\pi^c:X^c\to B'$ over $U'$, and which has central fiber an slc model,  is isomorphic to $\pi^c:X^c\to B'$ over $B'$.  The proof  is formally similar to the proof we sketched in the case of curves.  We direct the reader to  Koll\'ar [Pro.~6, Def.~7, Def.~15, and pp.8-9] for more details (see also \cite[Lem.~2.7]{bhattetal}). 
\end{proof}

\begin{rem} \label{remslcaut}
It is well known that a smooth, projective variety of general type has a  finite automorphism group (positive dimensional automorphism groups give rise to rational curves or abelian varieties covering the variety).   
More generally,   it is a result of Iitaka  that smooth, projective varieties of log general type 
have finite automorphism groups \cite[Lem.~1, p.87, Def.~p.71]{iitaka}.   Consequently, considering log resolutions of each irreducible component of the normalization, one would expect from \eqref{eqnlc} and Iitaka's result that an slc model would have a finite automorphism group; this is in fact the case \cite[p.328]{KSB}, \cite[Cor.~10.69]{kollarsing}, \cite[Lem.~2.5]{bhattetal}.   
\end{rem}

\begin{exa}
It is interesting to note the importance of having a condition such as slc in the remark above.  For instance, a plane quartic $C$  consisting of two smooth conics meeting in a single point,  which is a tacnode,   has ample canonical bundle  $\mathscr O_C(1)$.  However, the automorphism group of the curve is not finite (see \S \ref{secgitpc} below).  
\end{exa}

\begin{rem} 
Fix a Hilbert function $H:\mathbb Z\to \mathbb Z$, and let $\overline {\mathcal M}_H$ be the associated  category fibered in groupoids (over $\operatorname{Sch}_{\mathbb C}^{\text{\'et}}$), with objects that are families of slc models with Hilbert function $H$, as defined in \cite[Def.~29]{kollar11}, and with morphisms given by pull-back diagrams.   It is shown in \cite[Thm.~2.8]{bhattetal} (using recently announced results of Hacon--McKernan--Xu) that  $\overline{\mathcal M}_H$ is a proper   Deligne--Mumford $\mathbb C$-stack. 
\end{rem}


\section{Simultaneous stable reduction} \label{secsimsr}

Simultaneous stable reduction can be viewed as  the problem of stable reduction over bases other than a DVR; i.e.~extending families over higher dimensional bases.  
In the language of stacks, it is the problem of resolving rational maps from schemes to stacks.  Typically a generically finite base change is needed to do this.  The problem is in general quite delicate, and closely  related to the problem of resolving period maps to coarse moduli schemes.  

We discuss some well known results of Faltings--Chai \cite{fc} for abelian varieties, and some more recent results of de Jong, de Jong--Oort, and Cautis for curves.  An explicit example of simultaneous stable reduction was given in \S \ref{secexa}, and we start this section by discussing simultaneous stable reduction in the language of stacks.


\subsection{Simultaneous stable reduction in the language of stacks} \label{secssrs}

For a separated, finite type,  algebraic $Z$-stack $\mathcal M$, 
the valuative criterion for properness can be viewed as a question about resolving rational maps from the spectrums of DVRs into $\mathcal M$.  More precisely, $\mathcal M$ is proper if and only if for every DVR $R$, every rational map $S=\operatorname{Spec}R\dashrightarrow \mathcal M$ can be resolved after a generically finite base change.   

Simultaneous stable reduction is the problem of resolving rational maps from higher dimensional schemes into the stack.    
Let us make this more precise.   Given a $Z$-scheme $B$, a dense open subset $B^\circ\subseteq B$, and a $Z$-morphism
$$
B^{\circ}\to \mathcal M
$$
(which we will refer to loosely as a  rational $Z$-map $B\dashrightarrow \mathcal M$) we will say $B^\circ\to \mathcal M$ (or $B\dashrightarrow \mathcal M$) admits  a \textbf{simultaneous stable reduction}  if there exists a $Z$-alteration $\widetilde B\to B$ and a $Z$-morphism $\widetilde B\to \mathcal M$ extending the original morphism from $B^\circ$ in the sense that the following diagram is $2$-commutative:
\begin{equation}\label{eqnextssr}
\xymatrix{
\widetilde B \ar@{->}[d] \ar@{<-^)}[r]  \ar@{->}@/^1.2pc/[rr]& \widetilde B\times_B B^\circ\ar@{->}[d]  \ar@{->}[r] & \mathcal M\\
B  \ar@{<-^)}[r]& B^\circ \ar@{->}[ru]  &\\
}
\end{equation}

\begin{exa}
Let us review the example in \S \ref{secexa} in this language.  Recall we were given a family $X\to B=\mathbb A^2_k$ of plane cubics (with a section), and an open set $U\subseteq B$, so that the restriction $X_U\to U$ was a family of elliptic curves.  Such a family induces a morphism $U\to \overline{\mathcal M}_{1,1}$ (which we think of as a rational map $B\dashrightarrow \overline {\mathcal M}_{1,1}$).   We explicitly described  an alteration $\widetilde B'\to B$ with the property  that setting $\widetilde U'=U\times_B\widetilde B'$ to be the preimage,  the pull-back $X_U\times_U\widetilde U'$ extended to a family of stable marked curves $\widehat X\to \widetilde B'$.  This gives a morphism $\widetilde B'\to \overline {\mathcal M}_{1,1}$ extending the original morphism $U\to \overline {\mathcal M}_{1,1}$ (in the sense of \eqref{eqnextssr}).
\end{exa}

For proper, algebraic $Z$-stacks over a noetherian scheme $Z$, with finite diagonal,  simultaneous stable reductions exist quite generally.
The following result, which seems to be well known, was 
pointed out to the author by Fedorchuk (\cite[Rem.~7.3]{fedorchuk}).    

  \begin{teo}[{\cite[Rem.~7.3]{fedorchuk}, \cite[Thm.~2.7]{edidinetal}}]   \label{teossrfd}  
Let $Z$ be a noetherian scheme, and let $\mathcal M$ be a proper, algebraic $Z$-stack with finite diagonal   $\mathcal M\stackrel{\Delta}{\longrightarrow} \mathcal M\times_Z\mathcal M$. 
Then 
 any rational $Z$-map $B\dashrightarrow \mathcal M$ from a quasi-compact, quasi-separated  $Z$-scheme $B$  (or Noetherian $Z$-scheme $B$) admits a simultaneous stable reduction; i.e.~it can be resolved by an alteration.
\end{teo}

\begin{rem}\label{remdmfin}  
Recall that the diagonal morphism of an algebraic $Z$-stack is quasi-compact by assumption, and  an algebraic $Z$-stack is Deligne--Mumford if and only if the diagonal is unramified (e.g.~\cite[Thm.~8.1]{lmbchampes}).   
A quasi-compact, unramified morphism of schemes is quasi-finite.  
 In other words, Deligne--Mumford $Z$-stacks have quasi-finite diagonal (see also \cite[Lemma 1.13 (i)]{vistoli89}, and \cite[Rem.~2.5]{edidinetal} for a converse in characteristic $0$).     Recall also that an algebraic $Z$-stack locally of finite type has diagonal that is locally of finite presentation (e.g.~\cite[Cor.~1.4.3.1]{egaIV1}). 
 Finally, note that a separated stack has proper diagonal by definition; thus, since a proper, quasi-finite morphism, locally of finite presentation, is finite  (\cite[Thm.~8.11.1]{egaIV3}),  a separated, locally finite type algebraic $Z$-stack with quasi-finite diagonal in fact has finite diagonal.  In particular    
 a separated, finite type, Deligne--Mumford stack $\mathcal M/Z$ has finite diagonal (see also \cite[Lem.~1.13 (ii)]{vistoli89}).
\end{rem}

The theorem is an immediate consequence of the following lemma of Fedorchuk \cite{fedorchuk} and a theorem of  Edidin--Hassett--Kresch--Vistoli \cite{edidinetal}.

\begin{lem}[Fedorchuk {\cite[Rem.~7.3]{fedorchuk}}]  \label{lemfed} 
Let $Z$ be a noetherian scheme.  Let $\mathcal M$ be an algebraic  $Z$-stack, proper over $Z$,  that admits a finite, surjective $Z$-morphism 
$$
V\to \mathcal M
$$
from a scheme $V$.
Then any rational $Z$-map $B\dashrightarrow \mathcal M$ from a quasi-compact, quasi-separated  $Z$-scheme $B$ (or Noetherian $Z$-scheme $B$) admits a simultaneous stable reduction; i.e.~it can be resolved by an alteration.
\end{lem}

\begin{proof}  The proof (following Fedorchuk \cite{fedorchuk}) is short and we include it here. 
Consider the finite morphism 
$
V\to \mathcal M
$
assumed in the statement of the lemma.  Note we obtain that $V$ is proper over $Z$, since $V$ is finite (and hence proper) over $\mathcal M$ and we have assumed that $\mathcal M$ is proper over $Z$.  

Let $B^\circ\to \mathcal M$ be the morphism inducing the rational map $B\dashrightarrow \mathcal M$.  
From the definition of an algebraic stack, the diagonal is representable.  Consequently, $B^\circ\times_{\mathcal M}V$ is a scheme.
We then have
 a commutative diagram
\begin{equation} 
\xymatrix
{ 
B^\circ\times_{\mathcal M}V \ar@{->}[r] \ar@{->}[d]^{\text{finite}}& V \ar@{->}[d]^{\text{finite}}\\
B^\circ  \ar@{->}[r]& \mathcal M.\\
}
\end{equation}
Let $B'\to B$ be a finite morphism extending   $B^\circ\times_{\mathcal M}V\to B^\circ$; to obtain this extension one can either use Zariski's Main Theorem \cite[EGA  IV.3 Thm.~8.12.6,  p.45]{egaIV3} or \cite[Lem.~5.19, p.131]{fgae} (in the latter case, one extends the push forward of the structure sheaf of $B^\circ\times_{\mathcal M}V$ to a coherent sheaf on $B$ and then takes the relative spectrum).
 We thus obtain a rational map $B'\dashrightarrow V$.

Let $\widetilde B$ be the closure of the graph of $B^\circ\times_{\mathcal M}V\to V$ in $B'\times_Z V$.  The morphism $B'\times_Z V\to B'$ is proper by base change, and a closed immersion is proper.  It follows that $\widetilde B\to B'$ is proper, and also birational by construction. 
Thus the composition 
$$
\widetilde B\to B' \to B
$$ 
gives an alteration that resolves the map to $\mathcal M$.
\end{proof}

\begin{rem}
The assumption that the scheme $B$ be quasi-compact and quasi-separated, or that it be Noetherian, was used to ensure the existence of a finite cover of $B$ extending the given finite cover of $B^\circ$.  Another approach could be to assume that $B$ is covered by the spectrums of Japanese rings.  In this case, the appropriate integral closures will be finitely generated, allowing for another construction of a finite cover.
\end{rem}

Combining the lemma with the  following theorem of Edidin--Hassett--Kresch--Vistoli \cite{edidinetal} establishes Theorem \ref{teossrfd}.

\begin{teo}[Edidin et al.~{\cite[Thm.~2.7]{edidinetal}}]
  Suppose  $Z$ is  a noetherian scheme.  Let $\mathcal M$ be an algebraic $Z$-stack of finite type over $Z$.  Then the   diagonal 
$$
\mathcal M\to \mathcal M\times _Z \mathcal M
$$
is quasi-finite if and only if there exists a finite, surjective $Z$-morphism $V\to \mathcal M$ from a (not necessarily separated) $Z$-scheme $V$.
\end{teo} \label{teoehkv} 

From Theorem \ref{teossrfd} we obtain the following corollary.

\begin{cor}
Suppose that $\mathcal M$ is one of the following stacks: 
\begin{enumerate}

\item $\bar {\mathcal A}_g^A$, the moduli of stable semi-abelic pairs degree $1$ and dimension $g$;

\item $\overline {\mathcal M}_{g}$ ($g\ge 2$), the moduli of stable, genus $g$ curves;

\item $\overline {\mathcal M}^{ps}_g$ ($g\ge 3$), the moduli of pseudo-stable, genus $g$ curves;

\item $\overline {\mathcal M}_H$, the moduli of slc models  associated to a Hilbert function $H$;

\item $\overline{\mathcal P}_d$, the moduli of degree $d$, stable slc $K3$ pairs.

\end{enumerate}
 Then 
 any rational map $B\dashrightarrow \mathcal M$ from a quasi-compact, quasi-separated  scheme $B$  (or Noetherian  scheme $B$) admits a simultaneous stable reduction; i.e.~it can be resolved by an alteration. 
\end{cor}

\begin{rem}
In concrete terms, the corollary says the following. Given a dense open subset $U\subseteq B$, and a family $X_U\to U$, there exists an alteration $\widetilde B\to B$ so that the pull-back of the family can be extended to a family over all of $\widetilde B$.
\end{rem}

\begin{rem}
Part (2) of the corollary is a special case of a theorem of de Jong \cite[Thm.~5.8]{dejong1}.  We direct the reader there for more details, especially for a discussion of the total space of the family.  
\end{rem}

In many cases it can be useful to have an explicit description of an alteration giving a stable reduction.  We will call this an \textbf{explicit simultaneous stable reduction}.  Along these lines, one of the first questions one can ask is whether a rational map to a stack can be extended without an alteration.  In particular, when $B$ is non-singular, and $\Delta=B-B^\circ$ is an snc divisor,   a theorem giving conditions for the rational map to extend to $B$ will be called an \textbf{extension theorem}.

Finally, when a stack admits  a coarse moduli scheme, it can also be interesting to consider the problem of  resolving  the induced rational map to the coarse moduli scheme.    One place these types of problems show up naturally is in resolving rational (period) maps between coarse moduli schemes.   

More precisely, suppose  $\mathcal  M_1$ and $\mathcal M_2$  are algebraic $Z$-stacks admitting coarse moduli schemes $\mathcal M_1\to M_1$ and $\mathcal M_2\to M_2$.    Suppose there is an open dense subset $U_1\subseteq M_1$, which admits morphisms $U_1\to \mathcal M_1$ and $U_1\to \mathcal M_2$.  This induces a rational map $M_1\dashrightarrow M_2$, and one may be interested in  both a simultaneous stable reduction for $U_1\to \mathcal M_2$ as well as a resolution of the rational map $M_1\dashrightarrow M_2$.    We will consider both types of problems in what follows.

\subsection{Simultaneous stable reduction for abelian varieties}  We begin by considering extension theorems for abelian varieties.  That is we consider the case of extending families of abelian varieties over non-singular bases other than a DVR.   The main result we mention is due to Faltings--Chai \cite{fc}.

\begin{teo}[Faltings--Chai Extension {\cite[Thm.~6.7, p.185]{fc}}]  Let $B$ be a regular scheme over a field of characteristic $0$.  Let $
\Delta \subseteq B$ be an nc divisor.  Let $A_U$ be an abelian scheme over $U=B-\Delta$, which extends to a semi-abelian scheme $A_V$ over an open subscheme $V$ containing $U$ and the generic points of $\Delta$.  Then $A_U$ extends uniquely to a semi-abelian scheme $A_B$ over $B$.
\end{teo}

\begin{rem}
This fails in positive characteristic.  A counter example when the characteristic of the generic points of $B$ are positive is given in \cite[p.192]{fc}.   A counter example of Raynaud--Ogus--Gabber, when the characteristic of the generic points of $B$ are zero (but where other points have positive characteristic), is given in de~Jong--Oort \cite[\S 6]{dejongoort}.
\end{rem}

The Faltings--Chai theorem  implies a special case of the Borel Extension Theorem.  Recall that we use the notation $\mathcal A_g$ for the stack of principally polarized abelian varieties of dimension $g$.  A morphism $U\to\mathcal A_g$ corresponds to a  family $A_U\to U$ of principally polarized abelian varieties.  We denote the coarse moduli space by $A_g$.    We denote by $A_g^*$ the Satake (Bailly--Borel) compactification, and by $\bar A_g$ any one of Mumford's toroidal compactifications.  The most common toroidal compactification we will use is the second Voronoi, which we will denote by $\bar A_g^{Vor}$.  We direct the reader to \cite{namikawa80} for more details on compactifications of $A_g$.

\begin{teo}[Borel Extension {\cite[Thm.~A]{borelextension}}]
 Let $B$ be a regular scheme over a field of characteristic $0$.  Let $
\Delta \subseteq B$ be an nc divisor.  Setting $U=B-\Delta$, then for any morphism $f:U\to \mathcal A_g$, the composition $U\to \mathcal A_g\to A_g$ extends to a morphism $B\to A_{g}^*$.
\end{teo}

Borel's proof uses hyperbolic complex analysis and holds more generally for (locally liftable) holomorphic maps into Baily--Borel compactifications of arithmetic quotients of bounded symmetric domains.  Faltings--Chai \cite[Cor.~6.11, p.191]{fc}   also prove the related statement for maps into the moduli space $A_g[n]$ of principally polarized abelian varieties with level $n$-structure for $n\ge 3$.  In this case they can use  \cite[Thm.~7.9, Thm.~7.10, p.139]{git} to conclude that the coarse moduli space is quasi-projective and fine.  In other words, in both  this situation, as well as  under the assumptions of the Borel extension theorem as stated above, one may assume  there is a family of abelian varieties over $U$.  

The argument from there is short.  First, the extension statement is local.  One can also show that it suffices to establish extension after a finite base change (e.g.~\cite[Lem.~2.4]{cautis}).  
Thus we may take an 
\'etale base change, and assume we are in the situation where $B$ is regular and $\Delta$ has support defined by $x_1\cdots x_r$, where $x_1,\ldots,x_r$ form part of a system of local parameters.   Taking the finite cover $t_1=x_1^{m_1},\ldots,t_r=x_r^{m_r}$ for appropriate values of $m_1,\ldots,m_r$,  one uses the monodromy theorem to get extension over the generic points of $\Delta$.    The result then follows from the Faltings--Chai Extension Theorem.

\begin{rem}
The condition in the Borel Extension Theorem  that there is a family of abelian varieties over $U$ (or more generally that  the holomorphic map is locally liftable to the bounded symmetric domain) is essential.  More precisely,  for $B$ and $U$ as in the theorem, given a morphism $f:U\to A_g$, this need not  extend to a morphism $B\to A_{g}^*$.   An elementary example comes from the case of $A_1$ and $A_1^*$.  We can identify $A_1$ as the quotient $\mathbb H/\operatorname{SL}(2,\mathbb Z)$  of the upper half plane by the special linear group in the usual way, and it is well known that  $A_1^*$ is isomorphic to $\mathbb P^1_{\mathbb C}$.  The map $(\mathbb C^*)^2\to \mathbb P^1_{\mathbb C}$ given by  $(\lambda_1,\lambda_2)\mapsto [\lambda_1:\lambda_2]$ clearly does not extend to a morphism from $\mathbb C^2$.
\end{rem}

\begin{rem}
It is natural to ask whether a statement like the Borel extension theorem could hold for the toroidal compactifications of $A_g$, and indeed there is an extension theorem due to Ash--Mumford--Rapaport--Tai \cite{amrt} (see also Namikawa \cite[Thm.~7.29, p.78]{namikawa80}) giving explicit conditions for morphisms to extend over nc boundaries.  In concrete examples these extension conditions can be difficult to establish.   We discuss some particular examples below.
\end{rem}

\subsection{Examples of period maps to $A_g$}  We now consider the related problem of resolving period maps into compactifications of the moduli scheme of abelian varieties.     In this section we will work over $\mathbb C$.   The most well known example is  the Torelli map for curves; i.e.~the morphism 
$$
\mathcal T:\mathcal M_g\to \mathcal A_g
$$
that sends a curve $C$ to its principally polarized Jacobian $(JC,\Theta_C)$.    Let $T:M_g\to A_g$ be the associated morphism of coarse moduli spaces.  Torelli's theorem states that $T$ is injective.  

The boundary $\Delta$ in $\overline M_g$ is (up to finite quotient singularities) an nc divisor.  As a consequence of the Borel extension theorem, we obtain a morphism
$$
T^*:\overline M_g\to A_g^*
$$
extending $T$.  
For the toroidal compactifications of $A_g$, there are the  general extension results mentioned above.  In practice, these can be  difficult to verify.
It is a result of  Mumford and Namikawa \cite{namikawa80}, \cite[\S 18]{namikawa1}  that $T$ 
extends to a morphism
$$
\overline T^{Vor}:\overline M_g\to \bar A_g^{Vor}.
$$
Caporaso--Viviani describe the fibers of the morphism in \cite{capovivi}.
   In addition, it is shown in Alexeev \cite{alexeevtor} that there is a morphism $\overline {\mathcal T}^{Vor}:\overline{\mathcal M}_g\to \bar {\mathcal A}_g^{A}$ extending $\mathcal T$.  
   We direct the reader to Alexeev--Brunyate \cite{firstvor} for a proof that the Torelli map for stable curves extends to a morphism to the first Voronoi compactification (see also Gibney  \cite{gibtor} for more on the image of the  Torelli map to other toroidal compactificiations).

We now turn our attention to the Prym map.  We denote by $\mathcal R_g$ the moduli stack of connected, \'etale  double covers of non-singular curves of genus $g$.  
The Prym map
$$
\mathcal Pr:\mathcal R_g\to \mathcal A_{g-1}
$$
takes a double cover $\pi:\tilde C\to C$ to its principally polarized Prym variety $(P,\Xi)$ (see Mumford  \cite{mprym} for more details).  We denote by $Pr:R_g\to A_{g-1}$ the associated morphism of coarse moduli spaces.  It is well known that 
the map is dominant for $g\le 6$ (see esp.~\cite{bschott}), and in the other direction, 
Friedman--Smith \cite{fstor} and Kanev \cite{kanev} have shown that the map is generically injective for $g\ge 7$.

There is a compactification, $\overline {\mathcal R}_g$, due to Beauville \cite{bschott},  consisting of admissible double covers.   The coarse moduli space $\overline R_g$ has (up to finite quotient singularities) an nc boundary.  As a consequence, there is an extension 
$$
Pr^*:\overline R_g\to A_{g-1}^*.
$$
   On the other hand,  Friedman--Smith \cite{fs} have shown that the Prym map does not extend to a morphism  to $\bar A_{g-1}^{Vor}$ (or any reasonable Toroidal compactification). 
We direct the reader to Alexeev--Birkenhake--Hulek \cite{abh} for more details on the indeterminacy locus of the Prym map to $\bar A_{g-1}^{Vor}$.

The Clemens--Griffiths \cite{cg} period map for cubic threefolds provides another interesting example.
Recall that a cubic threefold is a smooth cubic hypersurface  $X\subseteq \mathbb P^4$.  The intermediate Jacobian of $X$ is the five dimensional complex torus  $JX:=
H^{1,2}(X)/H^3(X,\mathbb Z)$. This admits a principal polarization $\Theta_X$, given by the hermitian form $h$ on $H^{1,2}(X)$ defined by
$h(\alpha,\beta)=2i\int_X\alpha\wedge \bar{\beta}$.   Letting $M_{cub}$ be the moduli space of cubic threefolds, 
one obtains a morphism 
 $$J:M_{cub}\to A_5.$$
By virtue of the Clemens and Griffiths Torelli theorem \cite{cg} (see also Mumford \cite{mprym}),   $J$ is injective.  We denote the image by $I$, and we direct the reader to  Casalaina-Martin--Friedman \cite{cmf} for a geometric characterization of the abelian varieties parameterized by $I$.   

The space $M_{cub}$ admits a  GIT compactification
$$
\overline{M}_{cub}=\mathbb PH^0(\mathbb P^4,\mathscr O_{\mathbb P^4}(3))/\!\!/\operatorname{SL}(5),
$$
 (see Allcock \cite{allcock1}, Yokoyama \cite{yoko}) and  it is natural to consider extensions of the period map $J$ to $\overline M_{cub}$. 
Allcock--Carlson--Toledo \cite{act} and Looijenga--Swierstra \cite{looijengaswierstra1} have shown that $M_{cub}$ can be identified with an open dense subset of a 
ten dimensional ball quotient  $\mathcal B/\Gamma$.   They show moreover, that the rational map $\overline M_{cub}\dashrightarrow (\mathcal B/\Gamma)^*$ to the Baily--Borel compactification,  can be resolved by blowing up a single point.  We call the resulting space $\widehat M_{cub}$.
  
Using the description of $\widehat M_{cub}$ given in \cite{act,looijengaswierstra1},  Laza and the author describe an explicit blow-up $\widetilde M_{cub}$ of $\widehat M_{cub}$, with discriminant an nc divisor \cite{casalaza}.     The process used to obtain the resolution is the same as that described for simultaneous stable reduction for $ADE$ curves below (see \S \ref{secade}).    The Borel extension theorem then gives a morphism 
$$
J^*:\widetilde M_{cub}\to A_5^*.
$$
      Laza and the author  use this extension of the period map together with results from \cite{mprym, bschott, cmf, casa} and the explicit description of $\widetilde M_{cub}$ to describe the boundary of the image of $J^*$ \cite[Thm.~1.1]{casalaza}.

\begin{rem}  An explicit resolution of the map $\overline M_{cub} \dashrightarrow \bar A_5^{Vor}$ is still not known.  
    Certain components of the boundary of the (closure of the) image have been identified by Grushevsky--Salvati Manni \cite{samtriple} and Grushevsky--Hulek \cite{gruhulek} via the theory of theta functions.  
\end{rem}


\subsection{Simultaneous stable reduction for curves}
 We again begin by considering extension theorems.
In analogy with the Faltings--Chai extension theorem for abelian varieties, we mention the following extension   theorem  of de~Jong--Oort \cite{dejongoort}.

\begin{teo}[de~Jong--Oort Extension {\cite[Thm.~5.1] {dejongoort}}]\label{teodjo}
Let $B$ be a regular scheme and $\Delta$ an nc divisor on $B$.  Set $U=B-\Delta$.  A family of smooth curves of genus $g\ge 2$ over $U$ extends to a family of stable curves over $B$ if it extends to a family of stable curves over on open subset $V$ containing each generic point of $\Delta$.
\end{teo}

In fact the theorem is more general, in that one can allow for a generically stable family, so long as the topological type is locally constant on $U$.   A similar result was proven by Moret-Bailly \cite{mb}, where it is required that a generically smooth  family extend to a \emph{smooth}  family over the generic points of $\Delta$.   

A consequence of the de~Jong--Oort Extension Theorem is an analogue of the Borel Extension Theorem for stable curves.  Before stating the result, let us first rephrase the previous theorem in the language of stacks.  The theorem states that given a morphism to the stack $U\to \mathcal M_g$, there is an extension $B\to \overline {\mathcal M}_g$ if and only if there is an open set $V\subseteq B$ containing $U$ and the generic points of $\Delta$ and an extension $V\to \overline{\mathcal M}_g$.

\begin{cor}[Cautis {\cite[Thm.~A]{cautis}}] \label{corcaut} Let $B$ be a regular scheme and $\Delta$ an nc divisor on $B$.  Set $U=B-\Delta$.
Given a morphism $U\to \mathcal M_g$, there is an extension $B\to \overline {M}_g$.
\end{cor}

One obtains this corollary from the  previous theorem in the same way as the analogous statement was proven for semi-abelian varieties (i.e.~in the way the Borel Extension Theorem follows from the Faltings--Chai Extension Theorem).

An independent proof of the corollary was given by Cautis \cite[Thm.~A]{cautis}.  
By virtue of $\overline{\mathcal M}_g$ being a separated Deligne--Mumford stack, it is immediate to prove the de~Jong--Oort Extension Theorem from the corollary  using the Abramovich--Vistoli purity lemma \cite[Lemma 2.4.1]{av}.

\subsection{Explicit simultaneous stable reduction for curves} \label{secade}
Having established the existence of alterations resolving rational maps to $\overline {\mathcal M}_g$, one can ask for explicit alterations in specific settings.    One place where this type of question arises naturally is in the Hassett--Keel program for the moduli space of curves.  

We will not discuss the details of the Hassett--Keel program here (see \cite{hh}), but will simply note that in this program, projective varieties  $\overline{M}_g(\alpha)$, $\alpha\in [0,1]\cap \mathbb Q$ arise, which are conjectured to parameterize curves of genus $g$ with prescribed singularities (for $0\ll \alpha \le 1$ this has been established in Hassett--Hyeon \cite{hh,hh2}).    For ``most''  $g$ and $\alpha$ there are birational maps $$\overline{M}_g(\alpha)\dashrightarrow \overline {M}_g$$  to the moduli space of stable curves.  It would be of interest to have explicit resolutions.  In general, these birational maps will lift to rational maps to the stack $\overline{M}_g(\alpha)\dashrightarrow \overline{\mathcal M}_g$, and in this way  we arrive at the related problem of simultaneous stable reduction.

With this as motivation, we will consider the following problem.  \emph{Given a generically smooth family of curves $X\to B$ with fibers having prescribed singularities, give an explicit description of a resolution of the rational map $B\dashrightarrow \overline {\mathcal M}_g$}.

The specific case we will consider   is where the singular fibers have at worst $ADE$ singularities (we review the definition of $ADE$ singularities in \S \ref{secsing}). We call such curves $ADE$ curves, and we will consider the  question (\'etale) locally.

Laza and the author have given a solution to this problem in \cite{cml}  and Fedorchuk has given an independent solution  for  singularities of type $AD$ in \cite{fedorchuk}.  
Fedorchuk's proof is based on constructions of proper moduli spaces of hyperelliptic curves $\mathscr H[k,\ell]$, where the boundary consists of certain curves with $AD$ singularities  at worst of type  $A_k$ and $D_\ell$.  The proof provides modular descriptions of the spaces arising in the processes described below.  We direct the reader to Fedorchuk \cite{fedorchuk} for more details.

  Below is the version of the result in \cite{cml}.
Since we consider the resolution question (\'etale) locally,   it suffices to understand the case where $X\to B$ is a mini-versal deformation of an  $ADE$   curve $X_0$.    The statement of the theorem uses the notion of a Weyl cover, and wonderful blow-up; these  are explicit maps, which can be determined by the root systems associated to the singularities.   We refer the reader to  \cite[\S 2,3]{cml} for more details.  The Weyl cover and wonderful blow-up in \S \ref{secexa} are examples.   For the statement of the theorem, we note that the wonderful blow-up of the Weyl cover of $B$ has the property that the pull-back of the discriminant is an nc divisor, with irreducible components corresponding to curves with fixed singularity type.

\begin{teo}[Casalaina-Martin--Laza \cite{cml}, Fedorchuk \cite{fedorchuk}]\label{teocml}
 Let $ X \to B$ be a mini-versal deformation of an  $ADE$   curve $X_0$ with $p_a(X_0)=g\ge 2$.  The wonderful blow-up of the Weyl cover of $B$ resolves the rational moduli map to the moduli scheme $\overline{M}_g$, but fails to resolve the rational moduli map to the moduli stack $\overline{\mathcal M}_g$ along the $A_{2n}$ locus of the discriminant ($n\in \mathbb N$).  
The addition of a stack structure (generically $\mathbb Z/2\mathbb Z$ stabilizers) along this locus resolves the moduli map to $\overline{\mathcal M}_g$.
\end{teo}

\begin{rem}
 Let us elaborate on the final statement in the theorem concerning stacks.  There is a family of stable curves over the wonderful blow-up of the Weyl cover, except over the locus parameterizing curves with $A_{2n}$ singularities.  This locus is a collection of divisors, and there is an obstruction to extending the family over that locus.   At the generic points, the obstruction becomes trivial after taking a 
 branched double cover.
\end{rem}

We direct the reader to \cite{cml} for the proof.

 \begin{rem} As mentioned above in the section on period maps to the moduli space of abelian varieties, the method of proof of this theorem has applications to other situations  including the study of the moduli space of cubic threefolds \cite{casalaza}.
\end{rem}


\section{Simultaneous semi-stable reduction}
\label{secak}

The question of extending the Semi-stable Reduction Theorem to  higher dimensional bases is of course very natural, and was asked already in the introduction of \cite{mumetal}.    It has proven to be a difficult question; even the correct formulation of the problem is not immediately clear.
We  discuss some recent progress due to  de~Jong \cite{dejong1} and Abramovich--Karu \cite{ak}.

\subsection{A result of Abramovich--Karu}

The first issue to address is what is meant by semi-stable reduction for higher dimensional bases.  
We take the following modification of the assumptions in the statement of the Semi-stable Reduction Theorem as the starting point. \emph{We set $B$ to be an open subset of a non-singular variety, set $U\subseteq B$ to be a non-empty open subset and suppose that  
$
\pi:X\to B
$
is a surjective, projective morphism of a variety $X$ onto $B$ so that the restriction 
$
\pi_U:X|_U\to U
$
is smooth}.  

Our goal is to find  a diagram as in \eqref{eqnsrdiag} 
with $B'$ nonsingular, $f$ an alteration, $p$ a  projective modification, so that all of the geometric fibers of $\pi'$ satisfy some natural conditions.   For instance, at the very least, we would like all of the geometric fibers of $\pi'$ to be reduced.
Moreover, we could hope that all of the fibers have singularities that look at worst like smooth components meeting ``transversally''.

For instance,  when $\dim(X)=\dim(B)+1$, if one allows the total space $X'$ to be singular, then it is a result of de~Jong \cite[Thm.~5.8]{dejong1} that such a semi-stable reduction exists.  Moreover, de~Jong shows that in this case if $p$ is permitted to be an alteration, rather than a modification, then $X'$ can be taken to be smooth.   
However, for families with higher dimensional fibers, 
it is not possible in general to obtain a ``semi-stable reduction'' where the fibers all have singularities that at worst look like smooth components meeting transversally.
 For instance, the  two parameter family of surfaces defined by 
 \begin{equation}\label{eqnkaru}
(t_1-x_1x_2,t_2-x_3x_4)
\end{equation}
 precludes this (see Karu \cite[Exa.~1.12, p.21]{karuthesis}).  
Thus, in general, one needs a different definition of  ``semi-stable reduction'' to get a reasonable result.  

In light of the presentation in  \cite{mumetal}, and the family   \eqref{eqnkaru} above (which has fibers with at worst toric singularities), it is natural to change the focus to  toroidal structures.   Using this language, 
we state a  theorem of Abramovich--Karu, and then discuss some definitions after the theorem.

\begin{teo}[Abramovich--Karu {\cite[Thm.~0.3]{ak}}]
\label{T:AbrKar}
 Assume $\operatorname{char}(k)=0$ and $k=\bar k$. Let $X\to B$ be a surjective morphism of projective varieties over $k$, with geometrically integral generic fiber.  There exists a diagram as in \eqref{eqnsrdiag} 
with $X'$ a projective variety, $B'$ nonsingular, $f$ a projective alteration, $p$ a  projective strict modification, $\pi'$ a toroidal morphism,  and all of the geometric fibers of $\pi'$ equidimensional and reduced.
\end{teo}

A toroidal structure  on a normal variety $X$ is an open subset $U_X\subseteq X$, such that for each $x\in X$, there is a toric variety $X_{\sigma_x}$, a point $s\in X_{\sigma_x}$ and an isomorphism $\widehat {\mathscr O}_{X,x}\cong \widehat {\mathscr O}_{X_{\sigma_x},s}$ that maps the ideal of $X-U_X$ to the idea of $X_{\sigma_x}-T_{\sigma_x}$ where $T_{\sigma_x}$ is the torus of $X_{\sigma_x}$.    In other words, it is a variety together with an open set that \'etale locally looks like a toric variety together with its embedded torus.
A toroidal morphism is defined in the obvious way (see e.g.~\cite[Def.~1.3, p.247]{ak}).    We direct the reader to \cite[p.45]{ak} for the definition of a strict modification; we note that in the case that $X\to B$ is flat, $p$ will be a projective modification.

It is mentioned in \cite[Rem.~1.1]{ak} that it may also be  possible to address simultaneous semi-stable reduction using the language   of log-structures, rather than toroidal morphisms.   
We also direct the reader to \cite{ak2},  which addresses the case of schemes over  fields that are not algebraically closed. We conclude with the remark that roughly speaking, the theorem says that simultaneous semi-stable  reduction is possible if one allows for toric singularities.

\section{(Semi-)stable reduction for singularities}\label{secsing}

We now consider the (semi-)stable reduction problem locally and focus on singularities.  
The Mumford et al.~Semi-stable Reduction Theorem for one-parameter families ensures the existence of a semi-stable reduction for (generically smooth) one-parameter families of singularities.  The extensions to higher dimensional bases due to Abramovich--Karu  establish a certain form of existence in the simultaneous case.  Consequently, the  problem we will consider here is describing in more detail  semi-stable reductions for specific singularities.

Singularities of type $ADE$ will appear frequently in what follows.  Recall that these are the singularities (of dimension $n-1$, $n\ge 2$)  defined by the polynomials:
$$
\begin{array}{llll}
f_{A_k}&=& x_1^{k+1}+x_2^2+\ldots+x_n^k& k\ge 1\\
f_{D_k}&=&x_1(x_1^{k-2}+x_2^2)+x_3^2+\ldots +x_n^2& k\ge 4\\
f_{E_6}&=&x_1^4+x_2^3+x_3^2+\ldots +x_n^2 &\\
f_{E_7}&=&x_2(x_1^3+x_2^2)+x_3^2+\ldots +x_n^2 &\\
f_{E_8}&=&x_1^5+x_2^3+x_3^2+\ldots +x_n^2. &\\
\end{array}
$$

\subsection{Local stable reduction for curve singularities} \label{sechassett}
In this section we discuss some recent work of Hassett 
\cite{hassettstable} on local stable reduction  for isolated, locally planar singularities.  The main results are descriptions of the tails arising in the stable reduction process for curves.

\subsubsection{Preliminaries on local stable reduction} 
A local stable reduction of an isolated, plane curve singularity $(X_o,x)$ is defined as follows.  We consider 
$$
\pi:X\to B
$$ 
a one-parameter smoothing of $(X_o,x)$, with $X_o=\pi^{-1}(o)$ for some $o\in B$;  one can obtain such a smoothing by observing that the singularity $(X_o,x)$ will arise on some plane curve, and the Hilbert scheme containing that curve is a projective space with generic point parameterizing a smooth curve.   We then  perform semi-stable reduction following Mumford et al.~to obtain $\tilde {X}\to \tilde B$, where the central fiber is in nc position.  Set $p:\tilde {X}\to \tilde B\times_B  X$.
Finally,  take the log canonical model of $(\tilde {X},\tilde X_o)$ relative to the morphism $p$.  

 We will denote the resulting family by $ X^c\to \tilde B$;     this is called the \textbf{local stable reduction} of the family $X\to B$.    Note that $X^c$ agrees with $ \tilde B\times_B 
 X$ away from the central fiber.
By construction, the local stable reduction provides a local picture of the stable reduction for a one-parameter family of curves degenerating to a curve with a singularity $(X_o,x)$.  

We now review the definition of the tail of the local stable reduction.   The central fiber of $X^c\to \widetilde B$, which we will denote $X_o^c$, can be decomposed as $X^c_o=X_o^{\nu}\cup X_o^{T}$, where $X_o^{\nu}$ is the normalization of $X_o$ and $X_o^{T}:=\overline{X_o^c-X_o^{\nu}}$.  To fix notation, set  $X_o^{\nu}\cap X_o^{T}=\{p_1,\ldots.p_b\}$ where $b$ is the number of branches of $X_o$.  The pair $(X_o^{\nu}, \{p_1,\ldots,p_b\})$ depends only on $X_o$ and not on the choice of smoothing.  On the other hand,  the pair  $(X_o^{T},\{p_1,\ldots,p_b\})$ may depend on the smoothing, and we call this the \textbf{tail of the local stable reduction} of the family $X\to B$.   
 
\subsubsection{A result of Hassett}
 
We now mention Hassett's result that the tails arising in this process form subvarieties of the moduli space of curves.  We will use the notation $\overline M_{g,(n)}$ for the moduli space of stable curves of genus $g$, with $n$  unordered marked points.

\begin{pro}[Hassett {\cite[Prop.~3.2, p.176]{hassettstable}}]
Let $(X_o,x)$ be  a plane curve singularity with $b$ branches.  Let $\mathscr T_{X_o}$ be the set of tails obtained from the local stable reduction of each  smoothing of $X_o$.  The tails are connected, all of the same arithmetic genus $\gamma$, and  $\mathscr T_{X_o}$ is naturally a (reduced) subscheme of $\overline M_{\gamma,(b)}$.   
\end{pro}

In order to describe the subscheme $\mathscr T_{X_o}$ in more detail, Hassett considers the problem of deforming the pairs $(\operatorname{Spec}\mathbb C[[x,y]], X_o)$.  He considers  a process similar to that in the construction of the local stable reduction, performing semi-stable reduction for the pair $(\operatorname{Spec}\mathbb C[[x,y]], X_o)$ and then taking a log-canonical model.  His results give explicit descriptions of tails that arise in stable reduction for 
a wide class of singularities, including the classes known as  toric and quasi-toric singularities (which include $ADE$ singularities).  For the sake of space, we restrict to the  special case of $A_n$ singularities.

\begin{teo}[Hassett {\cite[Thm.~6.2,6.3, p.185-6, p.187, Prop.~7.3]{hassettstable}}]
Suppose that $(X_o,x)$ is a plane curve singularity of type $A_n$.  Then the scheme $\mathscr T_{X_o}$ is irreducible.  
\begin{enumerate}

\item If $n=2k$, then $\mathscr T_{X_o}$ is the closure of the locus of hyperelliptic curves of genus $k$, with a marked Weierstrass point in $\overline M_{k,1}$.  

\item If $n=2k+1$, then $\mathscr T_{X_o}$ is the closure of the locus of hyperelliptic curves of genus $k$ with two conjugate marked points (i.e.~interchanged by the hyperelliptic involution) in $\overline M_{k,(2)}$.   
\end{enumerate}
\end{teo}

\begin{rem}
One application of these results is to the Hassett--Keel program.  More precisely, the results can be used to provide a description of resolutions of rational maps among various moduli spaces that arise in the program.   We direct the reader to   \cite[\S 4.2]{cml} for more discussion (see also \S \ref{secade} above).
\end{rem}

\subsection{Simultaneous resolution for simple surface singularities and the Weyl cover}\label{secbrtj}
We now turn our attention to surface singularities, again over $\mathbb C$.  While in general one would consider questions of semi-stable reduction, for surface singularities of type $ADE$ there is a result due to Brieskorn--Tyurina that one may in fact find simultaneous resolutions of singularities.  

First let us recall what is meant by a simultaneous resolution of singularities.  Let $\pi:X\to B$ be a flat morphism of schemes.  A simultaneous resolution of singularities of $\pi$ is a commutative diagram
$$
\begin{CD}
X'@>p>> X\\
@V\pi'VV @V\pi VV\\
B@= B\\
\end{CD}
$$
such that $p$ is proper, $\pi'$ is smooth, and for every $b\in B$, the induced morphism $X'_b\to X_b$ is birational; i.e.~it is a coherent way of resolving the singularities of the fibers of $\pi$.  

Let us now make the following observation (\cite[Exa.~4.27, p.128]{kollarmori}):  If $B$ is a smooth curve, and 
$\pi$ is smooth over $B-\{o\}$ for some $o\in B$, then $\pi$ does not admit a simultaneous resolution if $X_{o}$ is a reduced singular curve, or  $\dim X_{o}\ge 3$ and $X_o$ is singular with at worst  isolated hypersurface singularities (see also Koll\'ar--Shepherd-Barron \cite{KSB}  for more on surfaces singularities and Friedman \cite{friedman86} for more on  threefolds). 
With this in mind, Brieskorn's theorem on surface singularities becomes quite surprising.

\begin{teo}[Brieskorn--Tyurina] Let $\pi:(X,x)\to (B,o)$ be a flat morphism of germs of singularities such that fiber $(X_o,x)$ is an $ADE$ surface singularity.  Then there is a finite, surjective morphism $(B',o')\to (B,o)$ such that $\pi':X':=B'\times_B X$ admits a simultaneous resolution of singularities.    
\end{teo}

We direct the reader  to Koll\'ar--Mori \cite[p.129]{kollarmori} for a  discussion of a number of techniques that can be used to prove the theorem, as well as some references.   
Brieskorn's \cite{briessing} Weyl group cover of the mini-versal deformation space of an $ADE$ singularity plays an important role.  We briefly review this now.
Let $X_o$ be an $ADE$ singularity of type $T$ (i.e.~$T=A_n$, $D_n$ or $E_n$).  Let $B_T$ be a mini-versal deformation space  of $X_o$ with discriminant $\Delta_T$.  Define $W_T$ to be the Weyl group of type $T$ and $R_T$ be the corresponding root system.  Brieskorn shows there exists a Galois cover $f:B_T'\to B_T$ with covering group $W_T$ and ramification locus $\Delta_T$ such that $f^*\Delta_T$ is an arrangement of hyperplanes determined by the root system $R_T$.  The hyperplanes are in one-to-one correspondence with the roots in $R_T$ considered up to $\pm 1$.  The morphism  $f:B_T'\to B_T$ is called the \textbf{Weyl (group)  cover}.

\begin{rem}
For surfaces, $ADE$ singularities are exactly the  canonical singularities (see e.g.~\cite{kollarmori}).  Thus this special case is enough to handle surface singularities in many circumstances.  We direct the reader to \cite{KSB} for a complete description of slc surface singularities.
\end{rem}

\section{Geometric invariant theory}  \label{secgit}

One way of determining a class of objects that will provide a stable reduction theorem for a moduli problem is via GIT.  
Typically, one will rigidify the moduli problem to obtain a (projective) fine moduli scheme, and then take the quotient by a reductive group to return to a (projective) scheme   parameterizing isomorphism classes of interest.  The GIT semi-stable points  naturally provide a class of objects where a weak form of stable reduction holds.  We call this  GIT semi-stable completion (or weak stable reduction for GIT) and discuss it in more detail in \S\ref{secsrgit}.

If there are no strictly semi-stable points, then one typically  obtains a stable reduction theorem.  
Note also that the stability conditions obtained in this way will depend on the rigidification, as well as the choice of linearization.  Different choices may lead to different stable reduction theorems.  We discuss this further in \S \ref{seccomparegit}.

\subsection{Preliminaries on GIT}  
  Let $X$ be a projective variety  over an algebraically closed field $k$.  Let $G$ be a linearly reductive algebraic group over $k$ \cite[Def.~1.4, p.26]{git} 
acting on $X$, and let $L$ be an ample $G$-linearized line bundle on $X$ \cite[Def.~1.6, p.30]{git}. 

For $n\in \mathbb N$ and a section $s\in H^0(X,L^{\otimes n})$, we set $$X_s=\{x\in X: s(x)\ne 0\}.$$
 Recall from \cite[Def.~1.7]{git} that the set of semi-stable (resp. stable, resp. properly stable) points of $X$, denoted $X^{ss}$ (resp. $X^s$, resp. $X^s_0$), is the set of points $x\in X$ such that there exists a natural number $n$ and a $G$-invariant section $s\in H^0(X,L^{\otimes n})^G$  with $s(x)\ne 0$ (resp. $s(x)\ne 0$ and the action of $G$ on $X_s$ closed, resp. $s(x)\ne 0$, the action of $G$ on $X_s$ closed, and the dimension of the stabilizer of $x$ is equal to $0$).    We denote the orbit of $x$ by $G\cdot x$ and the stabilizer of $x$ by $G_x$ \cite[p.3]{git}.

  Mumford's theorem  \cite[Theorem 1.10]{git} defines the GIT quotient of $X$ under the group action.  It states that there exists  a (surjective) universally submersive \cite[15.7.8, p.245]{egaIV3},  $G$-invariant  morphism of $k$-schemes
$$
\phi:X^{ss}\to X/\!\!/_LG
$$
that is  a categorical quotient of $X^{ss}$ by the action of $G$ \cite[Def.~0.5, p.3]{git}.  This satisfies the additional property that if $x_1$ and $x_2$ are closed points of $X^{ss}$, then $\phi(x_1)=\phi(x_2)$ if and only if 
$\overline {G\cdot x_1}\cap \overline {G\cdot x_2} \cap X^{ss}\ne \emptyset$ (\cite[p.40]{git}). 
In particular,  the closed points of $X/\!\!/_LG$ are in bijection with closed orbits of closed points in $X^{ss}$.  
There is an open subset $(X/\!\!/_LG)^\circ\subseteq X/\!\!/_LG$ with the property that $X^s_0=\phi^{-1}(X/\!\!/_LG)^\circ$ (\cite[(1) p.37]{git}),  
and the induced morphism 
$$\phi^\circ:X^s_0\to (X/\!\!/_LG)^\circ$$
 is a geometric quotient; in particular the fibers over closed points are exactly the orbits of the closed points of $X^s_0$ (see \cite[Def.~0.6, p.4]{git}).   It is also shown that 
$$
X/\!\!/_LG=\operatorname{Proj}\left(\bigoplus_{n=0}^\infty H^0(X,L^{\otimes n})^G\right)
$$
(see \cite[p.40]{git}, \cite[Prop.~8.1, p.120]{dolgachev}) 
so that $X/\!\!/_LG$ is projective.

\subsection{Weak stable reduction for GIT} \label{secsrgit}
As  $X/\!\!/_LG$ is projective, any map from the generic point of a DVR  to $X/\!\!/_LG$ extends to the whole DVR.  
We now consider the question of lifting such maps from $X/\!\!/_LG$ to $X^{ss}$.

More precisely, let $R$ be a DVR over $k$, with fraction field $K=K(R)$, and with residue field $\kappa(R)=k$.  Let $S=\operatorname{Spec} R$, let $\eta=\operatorname{Spec} K$ be the generic point, and let $s=\operatorname{Spec} \kappa(R)$ be the special point.  
We will assume we are given a morphism 
$$
f:S\to X/\!\!/_LG,
$$
and we are interested in lifting $f$ to $X^{ss}$.
The following result is often referred to as semi-stable completion for GIT (Shah \cite[Prop.~2.1, p.488]{shah}, Mumford \cite[Lem.~5.3, p.57]{mumfordstable}).

\begin{teo}[Weak stable reduction for GIT]
Let $f:S\to X/\!\!/_LG$ be a morphism.  
 There exists 
  a finite extension $K'$ of $K$, so that taking $R'$ to be the integral closure of $R$ in $K'$  and setting $S'=\operatorname{Spec}R'$, 
there is a  commutative diagram
\begin{equation} 
\xymatrix
{ 
 & & X^{ss} \ar@{->}^{\phi}[d]\\
S' \ar@{->}[r] \ar@{->}^{g}[rru] & S \ar@{->}_{f\ \ \ }[r]& X/\!\!/_LG.
}
\end{equation}
Moreover, $g$ may be chosen so that $g(s')$ lies in a closed orbit, where $s'$ is the closed point of $S'$.
\end{teo}

The essential point  is the universal submersiveness of $\phi$.  Indeed, if one were only to require that $S'\to S$  be a surjective morphism of spectrums of DVRs (and not necessarily a finite morphism), then the existence of such a lift $g$ would follow immediately  from the definition of universal submersiveness (see e.g.~\cite[Rem.~3.7.6, p.50]{kollarquot97} for a well-known converse).   The additional fact that $S'\to S$ can be taken to be finite follows from the proof of Mumford \cite[Lem., p.14]{git}.

\begin{rem} \label{remalperGIT}
Let us briefly consider weak stable reduction for GIT in the context of stacks.   
The analogous statement is that the natural map from the quotient stack $\pi:[X^{ss}/G]\to X/\!\!/_LG$  is universally closed.  While this can be established by modifying the proof of weak stable reduction for GIT,  it is also a consequence of the more general fact that  $\pi:[X^{ss}/G]\to X/\!\!/_LG$ is a  good (categorical) moduli space (\cite[Thm.~4.16 (ii), \S13, Thm.~6.6]{alpergms}).  
Concretely, given  
a map $f:S\to X/\!\!/_LG$ and a generic lift $g_\eta:\operatorname{Spec}K\to [X^{ss}/G]$, then after a generically finite base change, there is a lift $g:S'\to [X^{ss}/G]$ extending the pull back of $g_\eta$.    Moreover, one may choose the lift so that the closed point of $S'$ is sent to a closed point of $[X^{ss}/G]$ (corresponding to a closed orbit), and this point is unique (see \cite[\S 2]{asv10}).
\end{rem}

\begin{rem}  Since $X/\!\!/_LG$ is projective, it follows  that  $[X^{ss}/G]$ is universally closed and of finite type over $k$ (Remark \ref{remalperGIT}). It is not always the case, however,  that  $[X^{ss}/G]$ is separated.  For instance, this will fail if $X^{ss}\ne X^s_0$, as there will be positive dimensional affine stabilizers preventing the diagonal from being proper (see also \cite[Exa.~2.15]{asv10}).  On the other hand, if $[X^{ss}/G]$ is separated, it follows that it is also proper.  
Consequently, if $\mathcal M$ is a separated stack representing a moduli problem and $\mathcal M\cong [X^{ss}/G]$, then there is a  stable reduction theorem for the moduli problem.  
(See also \cite[Def.~2.1]{asv10} for the more general notions of weakly separated and weakly proper morphisms.)
\end{rem}

\subsection{GIT stable reduction for plane curves}\label{secgitpc}  In this section we consider the example of plane quartic curves, worked out by Mumford \cite[Ch.4 \S 2]{git}.  
To do this, we start with  the associated Hilbert scheme $X=\mathbb PH^0(\mathbb P^2,\mathscr O_{\mathbb P^2}(4))$.  There is a natural action of $\mathbb PGL(3)$ given by change of coordinates; as is typical, for the sake of simplicity we consider the action of $G=SL(3)$ instead via the isogeny $SL(3)\to \mathbb PGL(3)$.    The Hilbert scheme, being a projective space,  comes equipped with a polarization $L=\mathscr O(1)$ and a natural $SL(3)$-linearization.    We set $\overline M_3^{GIT}$ to be the GIT quotient 
$$
\overline M_3^{GIT}:=\mathbb PH^0(\mathbb P^2,\mathscr O_{\mathbb P^2}(4))/\!\!/_{\mathscr O(1)}SL(3)=X/\!\!/_{L} G
$$
Using the Hilbert--Mumford index, the following is worked out in \cite[p.81-2]{git} (see also \cite[Lem.~1.4]{artebani}).  
Let $C$ be a plane quartic corresponding to a point $x\in \mathbb PH^0(\mathbb P^2,\mathscr O_{\mathbb P^2}(4))=X$.

\begin{enumerate}

\item $x\in X^s_0$ if and only if $C$ is non-singular, or $C$ has only nodes and cusps as singularities.  

\item $x\in X^{ss}-X^s_0$ if and only if $C$ is a double conic or $C$ has a tacnode.  

\item $x\in X^{ss}-X^s_0$ \emph{and has closed orbit} if and only if $C$ is a double conic or $C$ is the union of two conics, at least one of which is non-singular, and the conics meet tangentially.

\end{enumerate}

While there is not a universal family of curves over $\overline M_3^{GIT}$, there is a universal family   over $X$, and 
the weak stable reduction theorem for GIT implies the following.  Given any one-parameter family of plane quartics over a punctured disk, with fibers of type (1)-(3) above, after a finite base change, the family can be filled in to a family over the complete disk, with central fiber of type (1) or (3).  Moreover, the isomorphism class of such a  central fiber is determined by the original family over the punctured disk.

\subsection{Deligne--Mumford stable reduction revisited}  \label{secgitgieseker}
\ \ Gieseker's    construction of the moduli space of stable curves as a GIT quotient of a  Hilbert scheme  provides another proof of the Deligne--Mumford stable reduction theorem \cite[p.i]{gieseker}.    
  
  Let $g\ge 2$ and $\nu\ge 10$.  Set $\operatorname{Hilb}_{g,\nu}$ to be the irreducible component of the Hilbert scheme containing the locus of $\nu$-canonically embedded, genus $g$, non-singular curves.   Let $H_{g,\nu}\subseteq \operatorname{Hilb}_{g,\nu}$ be the locus of (Deligne--Mumford) stable curves.
Set $N=(2\nu-1)(g-1)-1$ to be the dimension of $\nu$-canonical space.  The group $SL(N+1)$ acts on $\operatorname{Hilb}_{g,\nu}$ by change of basis.   Gieseker has shown (\cite[Ch.2]{gieseker}) that there exists an $SL(N+1)$-linearized polarization $\Lambda$ on $\operatorname{Hilb}_{g,\nu}$ such that
\begin{equation}\label{eqngieseker}
H_{g,\nu}=(\operatorname{Hilb}_{g,\nu})^s_0=\operatorname{Hilb}_{g,\nu}^{ss}.
\end{equation}   
Consequently,   one obtains   $\operatorname{Hilb}_{g,\nu}/\!\!/_\Lambda SL(N+1)\cong \overline M_g$ \cite[Thm.~2.0.2]{gieseker}.   A key point is the fact that  a family $X\to B$ of (Deligne--Mumford) stable curves over a scheme $B$   can (after possibly replacing $B$ by an appropriate open subset) be embedded in $\mathbb P^N_B$ as a flat family parameterized by a morphism $B\to \operatorname{Hilb}_{g,\nu}$ (e.g.~\cite[p.13]{gieseker}).   The Deligne--Mumford stable reduction theorem (over $k$, and up to the uniqueness statement) follows  from  \eqref{eqngieseker}  and  the weak stable reduction theorem for GIT.

\subsection{Comparing stability conditions}\label{seccomparegit}
We now compare the GIT stable reduction theorems arising from  \S \ref{secgitpc} and \S \ref{secgitgieseker}, and discuss the connection with the spaces arising in the Hassett--Keel program.  

Let us fix a DVR $R$, and consider a family $X\to S=\operatorname{Spec}R$ of smooth plane quartics degenerating to a quartic with a unique singularity, which is a tacnode ($A_3$).    This is a family of GIT semi-stable curves in the sense of \S \ref{secgitpc}.   However, the central fiber is not a curve with closed orbit, and the family is not a GIT semi-stable family 
in the sense of \S \ref{secgitgieseker}.  The GIT stable reduction theorem states that after a generically finite base change $S'\to S$, one can complete the family in two different ways.
In the sense of \S \ref{secgitpc}, one can complete the family so that the central fiber is the union of two plane conics meeting in two points, which are tacnodes (see \cite[\S3.4]{hl}).  In the sense of \S \ref{secgitgieseker}, one can complete the family so that the central fiber consists of a reducible stable curve obtained as the union of an elliptic curve (the normalization of the tacnodal quartic) attached to another elliptic curve (the tail; also called an elliptic bridge) at two points.

    In short, on the one hand, we are requiring the central fiber to be a plane quartic.  On the other, we are requiring the central fiber to be a nodal curve (with finite automorphisms).   Both conditions give a ``weak'' stable reduction theorem, albeit with very different central fibers.  We direct the reader to Hassett--Hyeon \cite{hh,hh2} for more discussion of GIT stability of curves with respect to different rigidifications, and linearizations.  See also Smyth \cite{smyth09}, Alper--Smyth--van der Wyck \cite{asv10} and Alper--Smyth \cite{as12} for a stack theoretic approach to this type of problem.

In terms of the Hassett--Keel program (see \S \ref{secade}),  the space $\overline M_3^{GIT}$ of \S \ref{secgitpc} is the space $\overline  M_3(17/28)$  \cite{hl}  (and $\overline M_3=\overline M_3(1)$).    The spaces are birational, as both contain dense open subsets corresponding to smooth, non-hyperelliptic curves.  In fact, the family of curves over the subset $U$ corresponding to  non-singular curves with trivial automorphism group induces a rational map $\overline M_3^{GIT}\dashrightarrow \overline{\mathcal M}_3$.  Resolving this map is closely related to the simultaneous stable reduction for curves with $ADE$ singularities, discussed in \S \ref{secade} (see  \cite[\S 8]{casalaza} and \cite[Cor.~3.6]{cml}).

\bibliography{srbib}

\end{document}